\documentclass[a4paper]{amsart}

\usepackage[latin1]{inputenc}
\usepackage[T1]{fontenc}
\DeclareFontFamily{T1}{wncyr}{}
\DeclareFontShape{T1}{wncyr}{m}{n}{%
  <5><6><7><8><9>gen*wncyr%
  <10><10.95><12><14.4><17.28><20.74><24.88>wncyr10}{}

\usepackage{calc}	    	
\usepackage{multicol}		
\usepackage{amssymb}	
\usepackage{wasysym}
\usepackage{verbatim}	
\usepackage{mathrsfs}	
\usepackage{dsfont}
\usepackage[dvipsnames]{xcolor}
\usepackage{pgf}
\usepackage{pgfpages}
\usepackage{tikz}
\usetikzlibrary{trees,shapes,%
  matrix,arrows,decorations.pathmorphing,backgrounds,positioning,fit}
\usepackage{indentfirst}
\usepackage{cancel}
\usepackage{fixltx2e} 
\usepackage{hyperref}

\theoremstyle{plain}

\newtheorem{thm}{Theorem}[section]
\newtheorem{lem}[thm]{Lemma}
\newtheorem{prop}[thm]{Proposition}
\newtheorem{coro}[thm]{Corollary}
\newtheorem*{thma}{Theorem}

\newtheorem*{claim}{Claim}

\theoremstyle{definition}

\newtheorem{defn}[thm]{Definition}
\newtheorem{exm}[thm]{Example}

\theoremstyle{remark}
\newtheorem{rem}[thm]{Remark}

\newcommand{\on}{\operatorname}

\newcommand{\mc}{\mathcal}
\newcommand{\mb}{\mathbf}
\newcommand{\mbb}{\mathbb}
\newcommand{\mf}{\mathfrak}

\newcommand{\id}{\ensuremath{\mathop{\rm id\,}\nolimits}}

\newcommand{\Z}{\mathbb{Z}}

\newcommand{\C}{\mathbb{C}}
\newcommand{\Q}{\mathbb{Q}}
\newcommand{\K}{\mathbb{K}}

\newcommand{\A}{\mathbb{A}}

\newcommand{\Sn}[1][n]{\mf S_{#1}}

\newcommand{\ve}{\varepsilon}

\newcommand{\ol}{\overline}

\newcommand{\mx}{\mbox}

\newcommand{\geqs}{\geqslant}
\newcommand{\leqs}{\leqslant}

\newcommand{\w}{\wedge}

\newcommand{\sm}{\setminus}
\newcommand{\sha}{\mathbin{\textup{\usefont{T1}{wncyr}{m}{n}\char120 }}}
\newcommand{\shap}{\textup{\scriptsize \usefont{T1}{wncyr}{m}{n}\char120 }}
\newcommand{\sh}{\operatorname{sh}}
\newcommand{\psh}{p_{\shap}}

\newcommand{\ra}{\rightarrow}
\newcommand{\lra}{\longrightarrow}

\newcommand{\Hom}{\on{Hom}}

\newcommand{\HH}{\on{H}}

\newcommand{\p}{\mathbb{P}}

\newcommand{\Sp}{\on{Spec}}
\newcommand{\Spec}{\Sp}

\newcommand{\dme}[1][]{\mc{DM}_{gm}^{eff}}
\newcommand{\dm}[1][]{\mc{DM}_{gm}}

\newcommand{\MTM}{\on{MTM}}

\newcommand{\SmCorr}[1][]{\on{\SmCorr}_{#1}}

\newcommand{\pst}[1]{{\p^1 #1 \setminus \{0,1,\infty\}} }
\newcommand{\ps}{{\pst{}}}

\newcounter{notemargin}
\newcommand{\Nc}[1][X]{\Nge{#1}{\bullet}}

\newcommand{\cNg}[2]{\Nge{#1}{#2}}

\newcommand{\Nge}[2]{{\mc N}^{qf, \,#2}_{#1}}
\newcommand{\Ngqf}[2]{{\mc N}^{qf, \,#2}_{#1}}

\newcommand{\Lcal}{\mc L}
\newcommand{\Lcz}{\Lcal_{0}^1}
\newcommand{\Lco}{\Lcal_{1}^0}
\newcommand{\Lc}{\Lcal^{0}}
\newcommand{\Lcu}{\Lcal^{1}}
\newcommand{\LcLcu}{\Lcal^{0-1}}
\newcommand{\Lce}{\Lcal^{\ve}}

\newcommand{\Lcb}{\Lcal^{B}}
\newcommand{\Lcub}{\Lcal^{1,B}}
\newcommand{\LcLcub}[1]{\ol{\Lcal^{0-1, B}_{#1}}}

\newcommand{\Zqf}[1][\bullet]{\mc Z_{q.f.}^{#1}}
\newcommand{\Alt}{\on{Alt}}

\newcommand{\ap}{{a'}}
\newcommand{\bp}{{b'}}

\newcommand{\dN}{\partial}
\newcommand{\da}[1][]{\partial_{#1}}

\newcommand{\dc}{d_{cy}}

\newcommand{\cL}{{\mc L^c}}
\newcommand{\Lie}{\on{Lie}}
\newcommand{\Ss}{s^{-1}}
\newcommand{\TC}{T^c}

\newcommand{\TL}[1][]{L_{#1}}
\newcommand{\Tcl}[1][]{\mc T^{coL}_{#1}}

\newcommand{\T}[1]{T_{#1^*}}

\newcommand{\Tc}[2][0]{T_{#2^*}^{#1}}
\newcommand{\Tcu}[1]{\Tc[1]{#1}}
\newcommand{\Tcg}{T}

\newlength{\ledge}
\newlength{\sibdis}
\newlength{\ecan}
\newlength{\ecas}
\newlength{\ecae}
\newlength{\ecao}
\newlength{\eca}
\newlength{\lbullet}
\tikzset{deftree/.style={level distance=\ledge,sibling distance=\sibdis}}
\tikzset{edgesp/.style={level distance=#1*\ledge,sibling distance=#1\sibdis}}
\tikzset{labf/.style={mathsc,yshift=-\eca}}
\tikzset{labfs/.style={mathss,yshift=-\eca}}
\tikzset{labr/.style={mathsc,yshift=\eca}}
\tikzset{labrs/.style={mathss,yshift=\eca}}
\tikzset{math mode/.style = {execute at begin node=$, execute at end node=$}}
\tikzset{mathscript mode/.style =%
 {execute at begin node=$\scriptstyle , execute at end node=$}}
\tikzset{math/.style = {execute at begin node=$, execute at end node=$}}
\tikzset{mathsc/.style =%
 {execute at begin node=$\scriptstyle , execute at end node=$}}
\tikzset{mathss/.style =%
 {execute at begin node=$\scriptscriptstyle , execute at end node=$}}
\tikzset{root/.style={draw,circle,inner sep=1pt,execute at begin node=$\bullet,
    execute at end node=$}}
\tikzset{roottest/.style={draw,circle,inner sep=#1pt}}
\tikzset{roots/.style={draw,circle,inner sep=2pt}}
\tikzset{bull/.style={fill,circle,minimum size=2pt,inner sep=0pt}}
\tikzset{leaf/.style={minimum size=2pt,inner sep=1pt}}
\tikzset{leafb/.style={minimum size=0pt,inner sep=0pt}}
\tikzset{intvertex/.style={mathsc,fill,circle,minimum size=0.6ex, inner sep=0pt}}
\tikzset{Reda/.style={-,double distance=0.3ex, draw=black}}
\tikzset{N/.style={-,thin,draw=black}}
\tikzset{Nd/.style={-,dotted, thin,draw=black}}

\title[A relative basis for $MTM(\ps)$]
{A relative basis for mixed Tate motives over the projective line minus
  three points}
\author{Ismael Soud{\`e}res}
\thanks{}
\address{Max-Planck-Insititut für Mathematik, Bonn\\
Vivatsgasse 7, \\ 53111 Bonn, \\
Germany \\
 souderes@mpim-bonn.mpg.de}
\date{\today}

\begin{document}
\thanks{None of this would
 have been possible without M. Levine patience, his explanations and deep
 inputs. The author is grateful to the MPIM in Bonn for providing ideal working
 conditions and support during my stay there as a guest where this work has been
 accomplished.}
\begin{abstract}
In a previous work, the author built two families of distinguished
algebraic cycles in Bloch-Kriz cubical cycle complex over the projective line
minus three points.

The goal of this paper is to show how these cycles induce well-defined elements
in the $\HH^0$ of the bar construction of the cycle complex and thus generate 
comodules over this $\HH^0$, that is a mixed Tate motives over the projective
line minus three points.

In addition, it is shown that out of the two families only one is needed at the
bar construction level. As a consequence, the author obtains that one of the
family gives a basis of the tannakian Lie coalgebra of mixed Tate motives over
$\ps$ relatively to the tannakian Lie coalgebra of mixed Tate motives over
$\Sp(\Q)$. This in turns provides a new formula for Goncharov
motivic coproduct, which really should be thought as a coaction. This new
presentation is explicitly controlled by the structure coefficients of Ihara action
by special derivation on the free Lie algebra on two generators.
 \end{abstract}
\maketitle
\tableofcontents
\section{Introduction}
\subsection{Multiple polylogarithms and mixed Tate motives}
For a tuple $(k_1,\ldots, k_m)$ of integers, the multiple polylogarithm is
defined by:
\[
Li_{k_1,\ldots,k_m}(z)=\sum_{n_1> \cdots >n_m}\frac{z^{n_1}}{n_1^{k_1} \cdots n_m^{k_m}}
\qquad (z \in \C, \, |z|<1).
\]
This is one of the one variable version of multiple polylogarithms in many
variables defined by Goncharov in  \cite{GSFGGon}.

When $k_1 \geqs 2$, the series converges as $z$ goes to $1$ and one recovers the
multiple zeta values
\[
\zeta(k_1,\ldots k_m)=
Li_{k_1,\ldots,k_m}(1)=\sum_{n_1> \cdots >n_m}\frac{1}{n_1^{k_1} \cdots n_m^{k_m}}.
\]
The case $m=1$ recovers the classical \emph{polylogarithm} and the value of
Riemann zeta function at $k_1$: $\zeta(k_1)$.

The values of multiple polylogarithms are important in geometric as they
naturally appear as periods, in the Hodge or motivic sense, of moduli spaces
of curves in genus $0$ (\cite{BrownMZVPMS}); as periods of the fundamental
groups of $\p^1$ minus a finite set of points (\cite{DG}).  In number theory,
Zagier's conjecture \cite{PDZKthZag} predicts that regulators of number fields
are linear combinations of polylogarithms at special points. 

Bloch and Kriz in \cite{BKMTM} constructed algebraic avatars of classical
polylogarithms. However this was part of a larger work proposing in 1994 a
tannakian category $MTM(F)$ of mixed Tate motives over a number field $F$. Their
construction begins with the cubical complex computing higher Chow groups which
in the case of $\Spec(F)$ is commutative differential graded algebra  $\mc
N_F$. The $\HH^0$ of the bar construction $B(\mc N_F)$ is a Hopf algebra and they
defined $MTM(F)$ as 
\[
MTM(F)=\mx{category of comodule over } \HH^0(\mc N_F).
\]
Spitzweck  in 
\cite{SpitzweckSCVTM} (as presented in \cite{KTMMLevine}) proved that this
construction agrees with Voevodsky definition of motives \cite{Vo00} and Levine's
approach to mixed Tate motive \cite{LevTM}. More recently, M. Levine generalized
this approach in \cite{LEVTMFG} to any quasi projective variety $X$ over the
spectrum of a field $\K$ such that
\begin{itemize}
\item the motive of $X$ is mixed Tate in Cisinski and Déglise category $DM(\K)$,
\item the motive of $X$ satisfies Beilinson-Soulé vanishing property.
\end{itemize}

In order to do so, Levine used the complex $\Nge{X}{\bullet}$ of quasi-finite
cycles over $X$ 
(Definition \ref{def:qfcycle}) instead of the original Bloch-Kriz complex. This
modification has better functoriality property and allows a simpler definition
of the product structure. Working over $\Spec(\Q)$, Levine's work shows that for
$X=\ps$: 
 
\begin{thm}[{\cite{LEVTMFG}[Section 6.6 and Corollary
    6.6.2]}]\label{pi1exactseq} Let $x$ be a 
  $\Q$-point of   $\ps$. Let $G_{\ps}$ and $G_{\Q}$ denote the Spectrum of
  $\HH^0(B(\Nge{\ps}{\bullet}))$ and $\HH^0(B(\Nge{\Q}{\bullet}))$ respectively.
 Then there is a tannakian category of mixed Tate motives over $\ps$:
\[
MTM(\ps)=\mx{category of comodule over } \HH^0(\Nge{\ps}{\bullet}).
\]
Moreover there is a split exact sequence:

 \begin{equation}\label{eq:motsespi1}
 \begin{tikzpicture}[baseline=(current bounding box.center)]
\matrix (m) [matrix of math nodes,
 row sep=0.6em, column sep=2.5em, 
 text height=1.5ex, text depth=0.25ex] 
 {1 & {\pi_1^{mot}(\ps,x)} & G_{\ps} &G_{\Q} & 1
 \\};
 \path[->,font=\scriptsize]
 (m-1-1) edge node[auto] {} (m-1-2)
(m-1-2) edge node[auto] {} (m-1-3)
(m-1-3) edge node[auto] {$p^*$} (m-1-4)
(m-1-4) edge node[auto] {} (m-1-5);
\path[->,font=\scriptsize, bend left]
(m-1-4) edge node[auto] {$x^*$} (m-1-3);
 \end{tikzpicture}
 \end{equation}
where $p$ is the structural morphism $p: \ps \lra \Sp(\Q)$. In the above
exact sequence ${\pi_1^{mot}(\ps,x)}$ denotes Deligne-Goncharov motivic fundamental
group \cite{DG}.
\end{thm}

The exact sequence \eqref{eq:motsespi1} is the motivic avatar of the short exact
sequence for etale fundamental groups. M. Levine however did not produce any
specific motives. In particular, he did not produce any specific element in
$\HH^0(B(\Nge{\ps}{\bullet}))$  ; a natural motive being then the comodule
cogenerated by such an element.
\subsection{Distinguished algebraic 
cycles over $\ps$}\label{subsec:alg-cycle-fam0}
In order to describe explicitly some elements in $\HH^0(B(\Nge{\ps}{\bullet}))$, a
first possible step is to produce a family of degree $1$ elements in $\Nge{\ps}{\bullet}$
which have a decomposable boundary inside the family. More explicitly, the
differential of such an element is a linear combination of products of other
elements inside the family. 

In \cite{SouMPCC} the author produces such a family. Together with two explicit
degree $1$ weight $1$ algebraic cycles generating the $\HH^1(\Nge{\ps}{\bullet})$,
the author obtains:
\begin{thma} 
 For any Lyndon word $W$ in the   letters $\{0,1\}$ of length $p\geqs 2$, there exists
  a cycle $\Lc_W$  in  $\cNg{\ps}{1}(p)$; i.e. a cycle of codimension
  $p$ in $\ps \times \A^{2p-1}\times A^p$ dominant and quasi-finite over $\ps
  \times \A^{2p-1}$. 

$\Lc_W$ satisfies :%
\begin{itemize}
\item $\Lc_W$ has a decomposable boundary,
\item  $\Lc_W$ admits an equidimensional extension to  $\A^1$ with
  empty fiber at $0$.
\end{itemize}
A similar statement holds for $1$ in place of $0$.
\end{thma}
The above result relies on
\begin{itemize}
\item The dual of the action of the free Lie algebra on two generators on itself by Ihara
  special derivations in order to ``guess'' the differential of cycles
  $\Lc_W$. 
\item The pull-back by the multiplication $\A^1 \times A^1 \lra \A^1$ in order
  to build the cycles $\Lc_W$ from their boundaries.
\end{itemize}
The free Lie algebra on two generators $\Lie(X_0,X_1)$ is the Lie algebra
associated to ${\pi_1^{mot}(\ps,x)}$ from the exact sequence
\eqref{eq:motsespi1} and hence appears naturally in the construction. However its graded
  dual $Q_{geom}$, which is more closely related to
  $\HH^0(B(\Nge{\ps}{\bullet}))$, is more natural in our context. It appears  in the
  sequence dual to \eqref{eq:motsespi1} :
\[
0 \lra Q_{\Q} \lra Q_{\ps} \lra Q_{geom} \lra 0
\]
where $Q_{\ps}$ and $Q_{\Q}$ denote respectively the set of indecomposable
elements of
$\HH^0(B(\Nge{\ps}{\bullet}))$ and $\HH^0(B(\Nge{\Spec(\Q)}{\bullet}))$.
\subsection{Main results}
In this paper, using the unit of the adjunction between bar and cobar
construction in the commutative/coLie case, we lift the above algebraic cycle to
elements in $Q_{\ps}$ viewed as a subspace of $\HH^0(B(\Nge{\ps}{\bullet}))$ by
the mean of Hain's projector $\psh$ (see \cite{HainIEBC} or Section
\ref{subsub:barcdga}). Let $\pi_1 : B(\Nge{\ps}{\bullet}) \lra \cNg{X}{\bullet}$
the projection onto the tensor degree $1$ part of the bar construction. We obtain 
\begin{thma}[{Theorem \ref{bareltX}}]
For any Lyndon word $W$ of length $p\geqs 2$ there exist elements $\Lcb_W$,
in the bar construction  $B(\Nge{\ps}{\bullet}$ satisfying:
\begin{itemize}
\item One has $\pi_1(\Lcb_W)=\Lc_W$. 
\item It is in the image of  $\psh$; that is in $Q_{\ps}$;
\item It is of degree $0$ and map to $0$ under the bar differential;
  it induces a class in $\HH^0(B_{\ps})$ and in $\HH^0(Q_{\ps})$.
\item Its cobracket in $Q_{\ps}$  is given  by the  differential of $\Lc_W$.
\end{itemize}
A similar statement holds for cycles $\Lcu_W$ and constant cycles $\Lc_W(1)$
induced by the fiber at $1$ of $\Lc_W$ (after extension to $\A^1$).
\end{thma}

Then we show that the elements $\Lcb_W$ and $\Lcub_W$ are related:
\begin{thma}[{Theorem \ref{LcLcuBA1}}]
For any Lyndon word $W$ of length $p\geqs 2$ the following relation holds
 in $\HH^0(Q_{X})=Q_{\HH^0(B_{X})}$
\begin{equation}
\Lcb_W-\Lcub_W= \Lcb_W(1).
\end{equation}
\end{thma}
The proof relies on the relation between the situation on $\ps$ and on $\A^1$
where the cohomology of $\Nge{\A^1}{\bullet}$ is given by constant cycle because of $\A^1$
homotopy invariance of higher Chow groups. As a corollary one obtains a
description of the cobracket of $\Lcb_W$  in terms of the structure coefficients
of Ihara action by special derivation. This makes explicit the relation between
the dual of Ihara action (or bracket) and
Goncharov motivic coproduct which here, as in Brown \cite{BrownMTMZ}, is in fact a
coaction. In a group setting, Goncharov coproduct is really the action  
of $G_{\Q}$ on $\pi_1^{mot}(\ps,x)$ induced by the short exact sequence
\eqref{eq:motsespi1}. 
 
We conclude the paper by showing at Theorem \ref{Lcbbasis} that the family of
elements $\Lcb_W$ induces a basis of $Q_{geom}$; that is a basis of $Q_{\ps}$
relatively to $Q_{\Q}$. 

Our methods are structural and geometric by opposition to Gangl Goncharov and
Levin approach \cite{GanGonLev05} toward lifting cycles to bar elements using
the combinatorics of ``rooted polygons''.  

The paper is organized as follow
\begin{itemize}
\item In the next section, Section \ref{sec:barcobar}, we begin by a short
  review of differential graded (dg) vector spaces and formalism. Then we
  present the bar cobar 
  adjunction in the case of associative algebras and coalgebras
  and in the case of commutative  algebras and Lie coalgebras.
\item In Section \ref{sec:baralgcycle}, we present the action of $\Lie(X_0,X_1)$ on
  itself by Ihara's  special derivations and  the corresponding Lie
  coalgebra. From there we recall the result from \cite{SouMPCC} constructing
  the cycles $\Lce_W$. We conclude this section by lifting the cycle to elements
  in the bar construction.
\item In section \ref{sec:relbar}, we prove that bar elements $\Lcb_W$ and
  $\Lcub_W$ are equal up to the constant (over $\ps$) bar element
  $\Lcb_W(1)$. From there we make explicit the relation with Ihara's coaction
  and prove that the elements $\Lcb_W$ provides a basis for $\Q_{geom}$ graded
  dual of the Lie algebra associated to $\pi_1^{mot}(\ps,x)$.
\end{itemize}
\section{Bar and cobar adjunctions}\label{sec:barcobar}
In this section we recall how the bar/cobar constructions gives a pair of
adjoint functors in the two following cases:
\[
B : \left \{\begin{array}{c}\text{diff. gr. ass. } \\ 
  \text{algebras} 
\end{array}\right \} 
\rightleftharpoons 
\left\{\begin{array}{c}\text{diff. gr. coass.} \\ 
\text{coalgebras} \end{array}\right\} : \Omega 
\]
and 
\[
B_{com} : \left \{\begin{array}{c}\text{diff. gr. com.} \\ 
  \text{ass. algebras} 
\end{array}\right \} 
\rightleftharpoons 
\left\{\begin{array}{c}\text{diff. gr. coLie.} \\ 
\text{coalgebras} \end{array}\right\} : \Omega_{coL} .
\]

As a differential graded commutative algebras $A$ is also an associative algebra,
we will recall how the two constructions are related in this case.

The material developed here is well known and can be found in Ginzburg-Kapranov
\cite{KDOGinKap} 
work even if their use of graded duals replaces coalgebra structures by algebra
structures.  The presentation used here is closer to the Kozul duality as
developed by Jones and Getzler in \cite{OHAGetJon}. We follow here the signs conventions
and the formalism presented by Loday and Vallette in \cite{LodValAO}. The
associative case is directly taken from
\cite[Chap. 2]{LodValAO} in a cohomological version. More about the  commutative/coLie
adjunction  can be found in \cite{OHAGetJon, KDOGinKap, MillesKCCTC, LCRHISinWal}.
\subsection{Notation and convention}
\subsubsection{Koszul sign rule} The objects are all object of the category of
(sign) \emph{graded $\Q$-vector spaces}. The degree of
an homogeneous element $v$ in $V$ is denoted by $|v|$ or $|v|_{V}$ if we want to
emphasis where $v$ is. The symmetric structure is given by the switching map
\[
\tau :V \otimes W \lra W \otimes V\, , \qquad
\tau(v\otimes w)=(-1)^{|v||w|} w\otimes v.
\]

For any maps $f: V \ra V'$ and $g : W \ra W'$ of graded spaces, the tensor
product
\[
f\otimes g : V \otimes W \lra V' \otimes W'
\]
is defined by
\[
(f\otimes g)(v\otimes w)=(-1)^{|g||v|}f(v)\otimes g(w).
\]

A \emph{differential graded} (dg) vector space is a
graded vector space equipped with a \emph{differential}  $d_V$ (or 
simply $d$) ; that is a degree $1$ linear map satisfying $d_V^2=0$. For $V$ and
$W$ two dg vector spaces the differential on $V\otimes W$ is defined by
\[
d_{V\otimes W} =\id_V\otimes d_W + d_V \otimes \id_W.
\]
\begin{defn}Let $V=\oplus_n V^n$ and $W=\oplus_n W^n$ be two graded vector
  spaces. A \emph{morphism of degree $r$}, say $f : V \lra W$,  is a collection
  of morphisms $f_n : V^n \lra W^{n+r}$. Let $\Hom(V,W)_r 
  $ be the vector space of morphisms of degree $r$.

When $V$ and $W$
are dg vector spaces, the graded vector space $\Hom(V,W)=\oplus \Hom(V,W)_r$
turns into a  a dg vector space with 
differential given by : 
\[
d_{\Hom}(f)=d_W\circ f - (-1)^r f\circ d_V
\]
for any homogeneous element $f$ of degree $r$. A  \emph{dg morphism} $f : V \lra
W$ is a morphism  satisfying $d_{\Hom}(f)=0$.  

\end{defn}

The dual of a graded vector space $V=\oplus_n V^n$ is defined by 
\[
V^*=\oplus_n \Hom_{Vect}(V^{-n},\Q)=\Hom(V,N)
\]
where the dg vector space $N$ is defined by $N=\Q$ concentrated in degree $0$
with $0$ differential. One has an obvious notion of cohomology on dg vector space.





\begin{defn}[(de)suspension]\label{suspensiondef} Let $S=s\Q$ be the $1$ dimensional
  dg vector space 
  concentrated in degree $1$ (that is $d_S=0$) generated by $s$.

The dual of $S$ is a one dimensional dg vector space  denoted by $S^{-1}$ and
generated by a degree $-1$ element $s^{-1}$ dual to $s$.

Let $V,d_V$ be a dg vector space. Its \emph{suspension} $(sV, d_{sV})$, is the dg
vector space $S\otimes V$. Its desuspension $(s^{-1}V,d_{s^{-1}V})$ is
$S^{-1}\otimes V$.
\end{defn}
There is a canonical identification $V^{n-1} \simeq (sV)^n $ given by
\[
i_s : V \lra sV \, \qquad v \longmapsto (-1)^{{|v|}_V}s\otimes v;
\]
under this identification $d_{sV}=-d_V$.  
\subsubsection{Associative dg algebra}

A \emph{differential graded associative} algebra $(A, d_A)$ abbreviated into dga
algebra  is a dg vector space equipped with a unital associative product
$\mu_A$ of degree $0$ commuting with the differential :
\[
d_A \circ \mu_A = \mu_A \circ d_{A \otimes A}
\]
and satisfying the usual commutative diagrams for an associative algebra, all
the maps involved being maps of dg vector spaces.

This last equality is nothing but Leibniz rule. The unit $1_A$ belongs to
$A^0$. On elements, one 
writes $a\cdot_{A} b$ or simply $a \cdot b$ instead of $\mu_A(a\otimes b)$.

\begin{defn}\label{defconnectedalg} The dga $A$ is \emph{connected} if $A^0=1_A \Q$.

\end{defn}

\begin{defn}[Tensor algebra]The tensor algebra over a dg vector space $V$ is
  defined by
\[
T(V)=\bigoplus_{n \geqs 0} V^{\otimes n}
\]
and equipped with the differential induced on each $V^{\otimes n}$ by $d_V$ and with
the concatenation product given by
\[
[a_1 | \cdots | a_n]\otimes [a_{n+1}|\cdots a_{n+m}] \longmapsto [a_1 | \cdots |
a_n|a_{n+1}|\cdots a_{n+m}] 
\]
where the ``bar'' notation $[a_1|\cdots |a_n]$ stands for $a_1\otimes \cdots
\otimes a_n$ in $V^{\otimes n}$.
\end{defn}

Note that the degree of $[a_1|\cdots |a_n]$ is $|a_1|_V + \cdots |a_n|_V$ and
that $T(V)$ admits a natural augmentation given by $\ve([a_1|\cdots |a_n])=0$ for
$n>0$ and  the convention $V^{\otimes 0}=\Q$. The concatenation product is
associative.

\subsubsection{Commutative and anti commutative algebras}
A commutative dga algebra $(A,d_A, \mu_A)$ or cdga algebra is a dga algebra such
that the multiplication commutes with the switching map:
\[
\begin{tikzpicture}
\matrix (m) [matrix of math nodes,
 row sep=2em, column sep=2em]
{A \otimes A & & A\otimes A  \\
 & A & \\
};
\path[->,font=\scriptsize]
(m-1-1) edge node[auto] {$\tau$} (m-1-3)
(m-1-1) edge node[below,xshift=-0.5ex] {$\mu_A$} (m-2-2);
\path[->,font=\scriptsize]
(m-1-3) edge node[auto] {$\mu_A$} (m-2-2);
\end{tikzpicture}
\] 
On homogeneous elements, this reads as 
\[
a\cdot b = (-1)^{|a||b|}b \cdot a.
\]

Let $V$ be a dg vector space and $n\geqs 1$ an integer.
The symmetric group $\mbb S_n$ acts on $V^{\otimes n}$ in two natural ways : the
symmetric action $\rho_S$ and the antisymmetric action $\rho_{\Lambda}$ (both graded).

For $i$ in $\{1, \ldots n-1\}$, let $\tau_i$ be the permutation exchanging $i$
and $i+1$. It is enough to define both actions for the $\tau_i$:
\[
\rho_S(\tau_i)=\underbrace{\id\otimes\cdots \id}_{i-1 \textrm{ factors}} \otimes \tau
\otimes \id \otimes \cdots 
\otimes \id 
\]
and
\[
\rho_{\Lambda}(\tau_i)=\underbrace{\id\otimes\cdots \id}_{i-1 \textrm{ factors}} \otimes
(-\tau) \otimes \id \otimes \cdots 
\otimes \id.
\]
where $\tau$ is the usual switching map. Both action involved signs. The graded
signature $\ve^{gr}(\sigma)\in \{\pm 1\}$ of a permutation
$\sigma$ is defined by
\[
\rho_S(\sigma)(v_1\otimes \cdots \otimes
v_n)=\ve^{gr}(\sigma)v_{\sigma^{-1}(1)}\otimes \cdots \otimes v_{\sigma^{-1}(n)}.
\]
Then, one has 
\[
\rho_{\Lambda}(\sigma)(v_1\otimes \cdots \otimes
v_n)=\ve(\sigma)\ve^{gr}(\sigma)v_{\sigma^{-1}(1)}\otimes \cdots \otimes v_{\sigma^{-1}(n)}
\]
where $\ve(\sigma)$ is the usual signature. 
Let $p_{S, n}$ be the projector defined on $V^{\otimes n}$ by
\[
p_{S, n}=\frac{1}{n!}\left( \sum_{\sigma \in \mbb S_n}\rho_S(\sigma)\right).
\]

\begin{defn} The (graded) symmetric algebra $S^{gr}(V)$  over 
  $V$
is defined as the quotient of $T(V)$ by the two side ideal generated by $(\id
-\tau)(a\otimes b)$.

One can write 
\[
S^{gr}(V)=\bigoplus_{n \geqs 0} S^{gr,n}(V)
\]
where $S^{gr, n}(V)=V^{\odot n}$ is the quotient of $V^{\otimes n}$ by the
symmetric action of $\mbb S_n$. 
\end{defn}
One also has the isomorphism $S^{gr,n}(V)=p_{S,n}(V^{\otimes n})$. $S^{gr}(V)$ is
the free commutative algebra over $V$. We may write simply $V\odot V$ for $S^{gr,2}(V)$.

\begin{defn} The (graded) antisymmetric algebra $\Lambda^{gr}(V)$ 
  $V$
is defined as the quotient of $T(V)$ by the two side ideal generated by $(\id
+\tau)(a\otimes b)$.

One can write 
\[
\Lambda^{gr}(V)=\bigoplus_{n \geqs 0} \Lambda^{gr, n}(V)
\]
where $\Lambda^{gr, n}(V)=V^{\w n}$ is the quotient of $V^{\otimes n}$ by the
antisymmetric action of $\mbb S_n$.  We may write $V\w V$ for $\Lambda^{gr,2}(V)$.
\end{defn}
As in the symmetric case, $\Lambda^{gr}(V)$ is also the image of $V^{\otimes n}$
by the projector $p_{\Lambda, n}=1/(n!)(\sum_{\sigma} \rho_{\Lambda}(\sigma))$
and $\Lambda(V)$ is the free antisymmetric algebra over $V$.

\subsubsection{Associative coalgebra}
A \emph{differential graded associative} coalgebra $(C, d_C)$ abbreviated into dga
coalgebra  is a dg vector space equipped with a counital associative coproduct
$\Delta_C$ of degree $0$ commuting with the differential :
\[
d_{C\otimes C} \circ \Delta_C = \Delta_C \circ d_{C}
\]
and satisfying the usual commutative diagrams for an associative coalgebra, all
the maps involved being maps of dg vector spaces. 

The iterated coproduct $\Delta^n : C \lra C^{\otimes (n+1)}$ is
\[
\Delta^n=(\Delta\otimes \id \otimes \cdots \otimes \id )\Delta^{n-1} \qquad ~ 
\qquad \mx{and} \qquad \Delta^1=\Delta.
\]
This definition is independent of the place of the
$\Delta$ factor (here in first position) because of the associativity of the coproduct.
We Will use  Sweedler's notation:
\[
\Delta(x)=\sum x_{(1)}\otimes x_{(2)}\, , \qquad 
(\Delta \otimes \id)\Delta(x)=\sum x_{(1)}\otimes x_{(2)} \otimes x_{(3)} 
\]
and 
\[
 \Delta^n(x)=\sum x_{(1)}\otimes \cdots \otimes x_{(n+1)}.
\]

A coaugmentation on $C$ is a
morphism of dga coalgebra $u : \Q \ra C$. In this case, $C$ is canonically
isomorphic to $\ker(\ve)\oplus \Q u(1)$. Let $\bar C=\ker(\ve)$ be the kernel of the
counit. 

When $C$ is coaugmented, the
\emph{reduced coproduct} is $\bar{\Delta}=\Delta-1\otimes \id - \id \otimes
1$. It is associative and there is an iterated reduced coproduct $\bar
\Delta^{n}$ for which we  also use Sweedler's notation.

\begin{defn}$C$ is \emph{conilpotent} when it is coaugmented and when, for any $x$ in
  $C$, one has $\bar \Delta^n(x)$ 
vanishes for $n$ large enough.
\end{defn} 
A \emph{cofree} associative dga coalgebra over the dg vector space is by
definition a \emph{conilpotent} dga coalgebra $F^c(V)$ equipped with a linear map
of degree $0$
$p : F^c(V) \lra V$ commuting with the differential such that $p(1)=0$. It 
factors any morphism of dg vector space $\phi : C \lra V$ where $C$ is a
conilpotent dga coalgebra with $\phi(1)=0$.

\begin{defn}[Tensor coalgebra] The tensor coalgebra over  $V$ is
  defined by
\[
T^c(V)=\bigoplus_{n \geqs 0} V^{\otimes n}
\]
and equipped with the differential induced on each $V^{\otimes n}$ by $d_V$ and with
the deconcatenation coproduct given by
\[
[a_1 | \cdots | a_n] \longmapsto \sum_{i=0}^{n+1} [a_1 | \cdots |a_i]\otimes
[a_{i+1}|\cdots a_{n}]. 
\]
\end{defn}
The deconcatenation coproduct is associative. The natural
projection $\pi_V :T^c(V) \lra \Q=V^{\otimes 0}$ onto the tensor degree $0$ part is a counit
for  $T^c(V)$ while the inclusion
$\Q=V^{\otimes 0} \lra T^c(V)$ gives the coaugmentation. The tensor coalgebra  $T^c(V)$
is the cofree counital 
dga coalgebra over $V$.

\subsubsection{dg Lie algebra}
We review here the definition of Lie algebra and Lie coalgebra in the dg
formalism. For any dg vector space $V$, let $\xi$ be the cyclic permutation of $V \otimes V
\otimes V$ defined by
\[
\xi=(\id \otimes \tau)(\tau \otimes \id).
\] 
It corresponds to the cycle sending $1$ to $3$, $3$ to $2$ and $2$ to $1$. 

\begin{defn}A dg Lie algebra $L$ is a dg vector space equipped with a degree $0$
  map of dg vector spaces $c : L \otimes L \lra L$ ($c$ stands for ``crochet'') satisfying
\[
c \circ \tau = -c \qquad \mx{and} \qquad 
c \circ (c\otimes \id) \circ (\id + \xi + \xi^2)  =0.
\]

On elements, we will use a bracket notation $[x,y]$ instead of $c(x\otimes y)$.
\end{defn} 
In the above definition, the first relation is the usual antisymmetry of the
bracket which gives in the dg context:
\[
[x,y]=(-1)^{|x||y|}[y,x].
\]
The second relation is the Jacobi relation:
\[
[[x,y],z]+(-1)^{|x|(|y|+|z|)}[[y,z],x]+(-1)^{|z|(|y|+|x|)}[[z,x],y]=0.
\]

One remarks that $(c\otimes \id)\circ
\xi = \tau \circ (\id \otimes c)$ and that  
$(c\otimes \id)\circ \xi^2=((c \circ \tau )\otimes \id)\circ (\id \otimes
\tau)$. Using this and the antisymmetry relation, one can rewrite the Jacobi
relation as a Leibniz relation:
\[
c\circ (c \otimes \id)=c\circ (\id\otimes c)+c\circ (c\otimes \id)\circ (\id
\otimes \tau).
\]

The definition of a dg Lie coalgebra is
dual to the definition of a Lie algebra.
\begin{defn}A dg Lie coalgebra $\cL$ is a dg vector space equipped with a degree $0$
  map of dg vector spaces $\delta : \cL  \lra \cL\otimes \cL$  satisfying
\[
\tau \circ \delta = -\delta \qquad \mx{and} \qquad 
(\id + \xi + \xi^2)\circ (\delta \otimes \id) \circ \delta  =0
\]
\end{defn} 

The first condition shows that $\delta$ induces a map (again denoted by $\delta$)
\[
\delta : \cL \lra \cL \w \cL.
\]
Let $\tau_{12}: \cL^{\otimes 3} \lra \cL^{\otimes 3}$ be the permutation
exchanging the two first factors. The second condition show that the following
diagram is commutative 
\[
\begin{tikzpicture}
\matrix (m) [matrix of math nodes,
 row sep=1em, column sep=5em, text height=1em, text depth=1ex]
{ & \cL \otimes \cL & \cL \otimes \cL \otimes \cL \\
\cL & & \\
& \Lambda^{gr,2}(\cL) & \Lambda^{gr,3}(\cL) \\
};
\path[->,font=\scriptsize]
(m-2-1) edge node[auto]{$\delta$} (m-1-2.west)
(m-2-1) edge node[auto]{$\delta$} (m-3-2.west)
(m-1-2) edge node[auto]{$\delta\otimes \id - \id \otimes \delta$} (m-1-3)
(m-3-2) edge node[auto]{$\delta\w \id - \id \w \delta$} (m-3-3)
(m-3-2) edge node[auto]{}(m-1-2)
(m-1-3) edge node[auto]{$1/6(\id-\tau_{12})(\id+\xi+\xi^2)$}(m-3-3);
\end{tikzpicture}
\]
%
and that the composition going through the bottom line is $0$. 
\subsection{Bar/cobar adjunction: associative case}
\subsubsection{Bar construction}
In this subsection, we recall briefly the bar/cobar construction and how they
give a pair of adjoint functor in the associative case.
\[
B : \left \{\begin{array}{c}\text{diff. gr. ass. } \\ 
  \text{algebras} 
\end{array}\right \} 
\rightleftharpoons 
\left\{\begin{array}{c}\text{diff. gr. coass.} \\ 
\text{coalgebras} \end{array}\right\} : \Omega.
\]

Let $(A, d_A, \mu_A, \ve_A)$ be an augmented dga algebra and $\bar A=\ker(\ve_A)$ its
augmentation ideal. The bar construction of $A$
is obtained by twisting the differential of the dga free coalgebra $T^c(\Ss \bar A)$.

The differential $d_A$ makes $\bar A$ and thus $\Ss \bar A$ into a dg vector vector
space. Let $D_1$ denote the induced differential on $T^c(\Ss \bar A)$ which in tensor
degree $n$ is:
\[
\sum_{i=1}^{n}\id^{i-1}\otimes d_{\Ss \bar A} \otimes \id^{n-i}.
\]

$S^{-1}=\Ss \Q$ admits an associative product-like map of degree $+1$
defined by:
\[
\Pi_s : \Ss\Q \otimes \Ss \Q \lra \Ss \Q \, \qquad \Pi_s(\Ss\otimes \Ss)=\Ss.
\] 
The map $\Pi_s$ and the restriction $\mu_{\bar A}$ of the multiplication $\mu_A$
to $\bar A$ induce the following map:
\[
f: \Ss \Q \otimes \bar A \otimes \Ss \Q \otimes \bar A 
\xrightarrow{\id\otimes \tau  \otimes \id}
\Ss \Q  \otimes \Ss \Q \otimes \bar A \otimes \bar A
\xrightarrow{\Pi_s\otimes \mu_{\bar A}} \Ss \Q \otimes \bar A.
\]

This map induces a degree $1$ map $D_2:T^c(\Ss \bar A) \lra T^c(\Ss \bar A)$
which satisfies $D_2^2=0$ because of the associativity of $\mu_A$

One check that the degree $1$ morphisms $D_1$ and $D_2$ commute (in the graded sense): 
\[
D_1\circ D_2+D_2 \circ D_1=0
\]

The coproduct on $\TC(\Ss \bar A)$ is given by the deconcatenation
coproduct. From these definitions, one obtains (see \cite{LodValAO}[Section
2.2.1]) the following.
\begin{lem}\label{bardef} The complex $B(A)=(\TC(\Ss \bar A),d_B)$ with
  $d_B=D_1+D_2$ and endowed with the deconcatenation coproduct $\Delta$ is a
  conilpotent dga coalgebra.
\end{lem}
We recall below the explicit formulas related to the bar construction $B(A)$:
\begin{itemize}
\item An homogeneous element $\mb a$ of tensor degree $n$ is denoted by
\[
[\Ss a_1|\cdots |\Ss a_n]
\]
or when the context is clear enough not to forget the shifting simply by 
$[a_1|\cdots |a_n]$. Its degree is given by:
\[
\deg_B(\mb a)=\sum_{i=1}^n\deg_{\Ss \bar A}(\Ss a_i)=\sum_{i=1}^n (\deg_A(a_i)-1)
\]
\item the coproduct is given by:
\[
\Delta(\mb a)=\sum_{i=1}^n[\Ss a_1|\cdots |\Ss a_i]\otimes [\Ss a_{i+1}|\cdots
|\Ss a_n].
\]
\item Let $\eta_{\mb a}(i)$ or simply $\eta(i)$ denote the ``partial degree'' of
  $\mb a$:
\[
\eta_{\mb a}(i)=\sum_{k=1}^i\deg_{\Ss \bar A}(\Ss a_k)=\sum_{k=1}^i(\deg_{ A}(a_k)-1).
\]
\item The differential $D_1$ and $D_2$ are explicitly given by the formulas:
\[D_1(\mb a)=-\sum_{i=1}^{n}(-1)^{\eta(i-1)}[\Ss a_1 | \cdots |\Ss d_{A}(a_i)|
\cdots | \Ss a_n]
\]
and
\[
D_2(\mb a)=-\sum_{i=1}^n(-1)^{\eta(i)}[\Ss a_1 | \cdots |\Ss \mu_{A}(a_i,a_{i+1})|
\cdots | \Ss a_n].
\]

The global minus sign in $D_1$ appears because the differential of the dg vector
space $\Ss \bar A$ is given by $d_{\Ss \bar A}(\Ss a)=-\Ss d_{A}(a)$. The other
signs are due to the Kozul sign rules taking care of the shifting.
\end{itemize}

\begin{rem}
This construction can be seen as a simplicial total complex associated 
to the complex $A$ (as in \cite{BKMTM}). Here, the augmentation makes it
possible to use 
directly $\bar A$  without referring to the tensor coalgebra over $A$ and
without the need of killing the degeneracies.
However the simplicial presentation usually masks the need of working with
the shifted complex; in particular for sign issues.

The bar construction $B(A)$ also admits a product $\sha$ which shuffles the
tensor factors. However, this extra structure becomes more interesting when $A$
is graded commutative and we will present it in the next section.
\end{rem}

The bar construction is a quasi-isomorphism invariant as shown in
\cite{LodValAO} (Proposition 2.2.4) and the construction provides a  functor:
\[
B: \left\{
 \text{aug. dga algebra}\right\} \lra \left\{ \text{coaug. dga
    coalgebra}
\right\}.
\]
\subsubsection{Cobar construction}
Analogously, one constructs the cobar functor. Let $(C, d_C, \Delta_C, \ve_C)$ be
a coaugmented dga 
coalgebra decomposed as $C=\bar C \oplus \Q$. Consider $T(s\bar C)$ the free
algebra over $s \bar C$ (with concatenation product). The differential on $C$
induces a differential $d_1$ on $T(s \bar C)$. $S=s\Q$ comes with a
coproduct-like degree $1$ map dual to $\Pi_s$:
\[
\Delta_s : s \Q \lra s \Q \otimes s \Q \, , \qquad \Delta_s(s)=-s\otimes s.
\]
The map $\Delta_s$ and the restriction of the reduced coproduct $\bar \Delta_c$
to $\bar C$ induce the following map:
\[
g: s \bar C \xrightarrow{\Delta_s \otimes \bar{\Delta}_C}
s\Q \otimes s \Q \otimes \bar C \otimes \bar C
\xrightarrow{\id \otimes \tau  \otimes \id}
s\Q \otimes \bar C\otimes s \Q  \otimes \bar C.
\]
It induces a degree $1$ map $d_2$ on $T(s \bar C)$ satisfying $d_2^2=0$
because of the coassociativity of $\Delta_c$. The two degree $1$ maps $d_1$ and
$d_2$ commute (in the graded sense):
\[
d_1 \circ d_2 +d_2 \circ d_1=0.
\]  

\begin{lem}\label{cobardef} The complex $\Omega(C)=(T(s \bar C),d_{\Omega})$ with
  $d_{\omega}=d_1+d_2$ and endowed with the concatenation product is an
  augmented dga algebra called the cobar construction of $C$.
\end{lem}
Note that the cobar construction is not in general a quasi-isomorphisms
invariant. The reader may look at \cite[Section 2.4]{LodValAO}  for more details. 

\subsubsection{Adjunction}
The two functors bar and cobar induces an adjunction described as follows:
\begin{thm}[{\cite[Theorem 2.2.9 and Corollary 2.3.4]{LodValAO}}]
For every augmented dga algebra $A$ and every conilpotent dga
coalgebra $C$ there exist natural bijections
\[
\Hom_{\text{\rm dga alg}}
(\Omega(C), A)\simeq %
 \on{Tw}(C, A) \simeq 
\Hom_{\text{\rm  dga coalg}}
(C, B(A)) .
\]


The unit $\upsilon : C \lra B \circ \Omega(C)$ and the  counit $\epsilon :
\Omega \circ B(A) \lra A$ are quasi-isomorphisms of dga coalgebras and algebras
respectively. 
\end{thm}
\subsection{Bar/cobar adjunction: commutative algebras/Lie Coalgebras}
In this section we recall the bar/cobar adjunction in the commutative/coLie
case giving a pair of functors:
 
\[
B_{com} : \left \{\begin{array}{c}\text{diff. gr. com.} \\ 
  \text{ass. algebras} 
\end{array}\right \} 
\rightleftharpoons 
\left\{\begin{array}{c}\text{diff. gr. coLie.} \\ 
\text{coalgebras} \end{array}\right\} : \Omega_{coL} .
\]

The cobar construction in the coLie case is a little more delicate. We will
concentrate on this construction. The bar construction in the commutative case,
will be presented  as the set of
indecomposable elements of the associative bar construction. A direct
construction can be found in \cite{LCRHISinWal}. Other descriptions were given in
\cite{OHAGetJon, KDOGinKap}. 
\subsubsection{Cobar construction for Lie coalgebras}\label{subs:cobarcoLie}
The construction follows follows the lines of the commutative dga coalgebra case.
However the lack of associativity and the use of the symmetric algebra need to
be taken into account.

First we need a notion of conilpotency for a Lie coalgebra
$(\cL,\delta,d)$. As $\delta$ is not associative, one can not directly use an
 iterated coproduct. One introduces trivalent trees controlling this lack of
 associativity. 

A rooted trivalent tree, or simply a tree, is a planar
tree (at each internal 
vertex a cyclic 
ordering of the incident edges is given) where vertices have valency $1$
(external vertices) or
$3$ (internal vertices) together with a distinguished external vertex (the
root); other external vertices are called leaves. The leaves are numbered from
left to right begin at $1$. The trees are drawn with the root (with number $0$)
at the top. 


Let $(\cL, \delta, d)$ be a dg Lie coalgebra. Recall that $\delta$ is a dg
morphism $\delta : \cL \lra \cL \otimes \cL$. 
\begin{defn}\label{deltaT}Let $T$ be a tree with $n$ as above and let $\{e_1,
  \ldots, e_n\}$ be 
  the set of its leaves ($e_i$ is the $i$-th leaf). $T$ induces a morphism 
\[
\delta_{T} : \cL \lra \cL^{\otimes n} 
\]
as follows:
\begin{itemize}
\item if $T$ has $n=1$ leaf, $\delta_T=\id_{\cL}$;
\item if $T$ has $n=2$ leaves, then $T=
\begin{tikzpicture}[%
baseline={(current bounding box.center)},scale=0.7]
\tikzstyle{every child node}=[intvertex]
\node[roots]
(r){}
[level distance=1.5em,sibling distance=3ex]
child {node{}[level distance=1.5em]
  child{ node(1){}}
  child{ node(2){}}
};
\fill (r.center) circle (1/0.7*\lbullet) ;
\node[mathsc, xshift=-1ex] at (r.west) {};
\node[labf] at (1.south){e_1};
\node[labf,xshift=0ex] at (2.south){e_{2}};
\end{tikzpicture}$ and
  $\delta_T=\delta$; 
\item if $T$ has $n\geqs 3$ leaves, then there exists at least one leaf $e_i$ in
  a strict subtree of the form 
\[
T_{0}=\begin{tikzpicture}[%
baseline={(current bounding box.center)},scale=1]
\tikzstyle{every child node}=[intvertex]
\node[intvertex]
(r){}
[level distance=1.5em,sibling distance=3ex]
child {node{}[level distance=1.5em]
  child{ node(1){}}
  child{ node(2){}}
};
\node[mathsc, xshift=-1ex] at (r.west) {v};
\node[labf] at (1.south){e_i};
\node[labf,xshift=1ex] at (2.south){e_{i+1}};
\end{tikzpicture}
\]
where $v$ is an internal vertex of $T$. Let $T'$ be the tree $T\sm T_0$, where
the internal vertex $v$ of $T$ is the $i$-th leaf (in $T'$). The morphism
$\delta_T$ is defined by
\[
\delta_T=(\id^{\otimes (i-1)}\otimes \delta \otimes \id^{\otimes (n-i)})\circ \delta_{T'}.
\]
\end{itemize}
\end{defn}

This definition does not depend on the choice of the subtree $T_0$. By analogy with the
associative case, we define.
\begin{defn} A Lie coalgebra $(\cL, \delta, d)$ is \emph{conilpotent} if for any $x
  \in \cL$ there exist $n$ big enough such that for any tree $T$ with $k$
  leaves, $k\geqs n$, $\delta_T(x)=0$.
\end{defn}

We now fix a conilpotent dg Lie coalgebra $(\cL,\delta, d_{\cL})$. Its cobar
construction is given by twisting the differential of the free commutative dga
$S^{gr}(s \cL)$.

The differential $d_{\cL}$ induces a differential $d_{s\cL}$ on $s\cL$ and thus
on $(s\cL)^{\otimes n}$ given by
\[
\sum_{i=1}^n id^{\otimes (i-1)}\otimes d_{s\cL}\otimes \id^{\otimes (n-i)} :
(s\cL)^{\otimes n} \lra (s\cL)^{\otimes n}.
\]
This differential goes down to a differential on $n$-th symmetric power of $s\cL$:
\[
D_1 : S^{gr,n}(s\cL) \lra S^{gr, n}(s\cL).
\] 

Using the map $\Delta_s$ and the cobracket $\delta_{\cL}$, one has a morphism:
\[
g_{L}:
s\cL \xrightarrow{\Delta_s\otimes \cL} s\Q \otimes s\Q \otimes \cL \otimes \cL
\xrightarrow{\id \otimes \tau \otimes id} s\Q \otimes \cL \otimes s \Q \otimes \cL.
\]
which induces a degree $1$ map $\delta^s : s\cL \lra S^{2,gr}(s\cL)$ because of
the relation $ \tau \circ \delta = -\delta $ and the sift in the degree.


The relation $(\id + \xi + \xi^2)\circ (\delta \otimes \id) \circ \delta  =0$,
combined with the shift in the degree and $\Delta_s$, shows that  $g_L$ induce a
differential $D_2$ on 
$S^{gr}(s\cL)$ given  by the formula:
\[
D_2|_{S^{gr,n}(s\cL)}=\sum_{i=1}^{n}\id^{\otimes(i-1)}\otimes g_L \otimes \id^{\otimes(n-i)}.
\] 
This is the
classical duality between Jacobi identity and $D^2=0$ for a classical Lie
coalgebra (that is with a dg structure concentrated in degree $0$).

The differential $D_1$ and $D_2$ commute (in the graded sens), that is 
\[
D_1 \circ D_2 + D_2 \circ D_1 =0
\]
 and one obtains the following.
\begin{lem}\label{cobarcoLie}The complex $\Omega_{coL}(\cL) = (S^{gr}
  (s\cL),d_{\Omega, coL})$
with the symmetric concatenation product is an augmented  commutative dga
algebra called the cobar-coLie  construction of $\cL$.
\end{lem}

\subsubsection{Bar construction for commutative dga algebras}\label{subsub:barcdga}
Let $(A, d_A , \mu_A , \ve_A )$ be an augmented commutative dga algebra and
$\bar A=\ker(\ve_A)$. One can consider its bar construction $B(A)$ as
associative algebra. One defines on the coalgebra $B(A)$ an associative
 product $\sha$ by the formula
\begin{align*}
[x_1|\cdots | x_n]\sha [x_{n+1}|\cdots|x_{n+m}]&= 
\sum_{\sigma \in \sh(n,m)}\rho_{S}(\sigma)([x_1|\cdots | x_{n}|x_{n+1}|\cdots
|x_{n+m}]) \\
&=\sum_{\sigma \in \sh(n,m)}\ve^{gr}(\sigma)([x_{\sigma^{-1}(1)}|\cdots |
x_{\sigma^{-1}(n+m)}]) 
\end{align*} 
where $\sh(n,m)$ denotes the subset $\mbb S_{n+m}$ preserving the order of the ordered sets 
$\{1, \ldots, n\}$ and $\{n+1,\ldots, n+m\}$. For grading reasons, and thus for
signs issues, it is important to note that, in the above formula,
the $x_i$'s are elements of $s^{-1}A$.

A direct computation shows that $\sha$ turns $B(A)$ into an augmented commutative
dga Hopf algebra; that is it respects the expected diagrams for a Hopf algebra,
all the morphisms involved being dg-morphisms. 

The augmentation of $B(A)$ is the projection onto the tensor degree $0$ part. 
Let $\ol{B(A)}$ be the kernel of the augmentation of $B(A)$ and  
\[
Q_B(A)=\ol{B(A)}/({\ol{(B(A)})}^2
\]
be the set of its indecomposable elements.  Ree's
theorem \cite[Theorem 1.3.9]{LodValAO} (originally in
\cite{ReeLEAS}) 
shows the following
\begin{lem}The
differential $d_B$ induces a differential $d_Q$ on $Q_B(A)$.
The reduced coproduct $\Delta'$
induces a cobracket $\delta_Q=1/2(\bar \Delta - \tau \bar \Delta)$
 on $Q_B(A)$ making it into a conilpotent Lie dg coalgebra. 

The complex $(Q_B(A), d_Q)$ endowed with the cobracket $\delta_Q$ is the
\emph{commutative bar construction} of $A$ and denoted $B_{com}(A)$.
\end{lem}

Working with the indecomposable elements as a quotient may be complicated. In
particular, some structure, say 
for example extra filtrations,  may not behave well by taking a
quotient. For this purpose,  R. Hain, dealing with Hodge structure problems,
gave in \cite{HainIEBC} a splitting 
\[
i_Q : Q_B(A) \lra \ol{B(A)}
\] of the projection $p_Q : \ol{B(A)} \lra Q_B(A)$ commuting with the
differential. The projector of $\ol{B(A)}$ given 
by the composition $\psh=i_Q \circ p_Q$ can be express using the following
explicit formula given in \cite{HainIEBC}:
\[
\psh([a_1|\cdots|a_n])=\sum_{i=1}^n \frac{(-1)^{i-1}}{i} \sha \circ
(\bar \Delta)^{i-1} 
([a_1|\cdots | a_n])
\]
where the associative product $\sha$ has been extend to $B(A)^{\otimes n}$ for
all $n\geqs 2$ and where $\bar{\Delta}^{(0)}=\id$.
\subsubsection{Adjunction}  As in the case of associative algebras and
coalgebras, the functors $\Omega_{coL}$ and $B_{com}$ are adjoint.
\begin{thm}
For any augmented commutative dga algebra $A$ and any conilpotent dg Lie
coalgebra $\cL$ there exist natural bijections
\[
\Hom_{\text{\rm com dga alg}}
(\Omega_{coL}(\cL), A)
\simeq \Hom_{\text{\rm conil dg coLie}}
(\cL, B_{com}(A)) .
\]


The unit $\upsilon : \cL \lra B_{com} \circ \Omega_{coL}(\cL)$ and the  counit $\epsilon :
\Omega_{coL} \circ B_{com}(A) \lra A$ are quasi-isomorphisms of conilpotent dg
co Lie algebras and commutative dga algebras respectively. 
\end{thm}

Note that the following diagram of dg vector spaces is commutative:
\[
\begin{tikzpicture}
\matrix (m) [matrix of math nodes,
 row sep=3em, column sep=10em]
{\cL  & \Omega_{coL}(\cL)=\left(S^{gr}(s\cL), d_{\Omega,coL}\right)  \\
 & B_{com}(\Omega_{coL}(\cL))=Q_B\left(S^{gr}(s\cL), d_{\Omega,coL}\right) \\
};
\node[xshift=-5em] (u) at (m-2-2.north) {};
\node[xshift=-5em] (v) at (m-1-2.south) {};
\path[->,font=\scriptsize]
(m-1-1) edge node[auto] {$\upsilon$} (m-2-2.west)
(m-1-1) edge node[below,xshift=-0.5ex] {} (m-1-2)
(u) edge node[auto] {$\pi_1$} (v);
\end{tikzpicture}
\]
where $\pi_1$ is the projection onto the tensor degree $1$ part restricted to
the set of 
indecomposable elements; that is restricted to the image of $\psh$. 
\subsection{An explicit map}
We present here an explicit
description for the unit $\upsilon : \cL \lra B_{com} \circ \Omega_{coL}(\cL)$
with image in $\TC(S^{gr}(s\cL))$ using the projector $\psh$ onto the
indecomposable elements.

\begin{defn}Let $n$ be a positive integer. \begin{itemize}
\item The Catalan number $C(n-1)$ gives the number of rooted trivalent trees
  with $n$ leaves.
\item Let $T$ be such a rooted trivalent tree with $n$ leaves. Define
  $\tilde{\delta}_T$ by 
\[
\tilde{\delta}_T=\frac{1}{2^n}\delta_T
\]
where the map $\delta_T : \cL \lra \cL^{\otimes n}$ was introduced at Definition
\ref{deltaT}. 
\item The morphism $\tilde{\delta}_n$ is up to a normalizing coefficient the
  sums for all trees $T$ with $n$ 
  leaves of the morphism $\tilde{\Delta}_T$:
\[
\tilde{\Delta}_n=\sum_{T }
\frac{1}{nC(n-1)}\tilde{\Delta}_T.
\]
where the sums runs through all trivalent tree with $|T|=n$ leaves. 
\end{itemize}
\end{defn}

Note that using the identification $\cL \simeq s^{-1}\otimes s \otimes \cL$
given by $x \mapsto s^{-1}\otimes s \otimes x$, the
morphism $\tilde{\Delta}_n$ induces a morphism 
\[
\cL \lra  \left(s^{-1}\otimes s \otimes
\cL\right)^{\otimes n}
\]
again denoted by $\tilde{\Delta}_n$.
\begin{claim} The composition
\[ 
\phi_{\cL} =\psh \circ \left(\sum_{n\geqs 1} \tilde{\Delta_n}\right)
\]
gives a morphism 
\[
\phi_{\cL} : \cL \lra  B_{com} \circ \Omega_{coL}(\cL)=Q_B\left(S^{gr}(s\cL)\right)
\]
which is equal to the unit of the adjunction:
\[
\phi_{\cL}=\upsilon.
\]
In the above identification $
\Omega_{coL}(\cL)=Q_B\left(S^{gr}(s\cL)\right)$, the differential on
$\left(S^{gr}(s\cL)\right)$ is the bar differential $d_{\Omega, coL})$.
\end{claim}


The idea for this formula comes from mimicking the associative case where the
unit morphism being a coalgebra morphisms has to be compatible with the iterated
reduced coproduct.
  This formula can be derived as a consequence
  of the work of Getzler and Jones \cite{OHAGetJon}  or of the work of Sinha and
  Walter \cite{SinWalLCRHTI}.
\section{Families of bar elements}\label{sec:baralgcycle}
In \cite{SouMPCC}, the author defined a family of algebraic cycles $\mc
L^{\ve}_W$ indexed by couple $(W, \ve)$ where $W$ is a Lyndon word and $\ve$ is in
$\{0 ,1\}$. Note that, all words considered in this work are words in
the two letters $0$ and $1$.

One of the idea underling the construction of the cycles was to
follow explicitly a $1$-minimal model construction described in \cite{DGMS} with the
hope  to use the relation between $1$-minimal model and bar construction in 
order to obtain motives over $\ps$ in the sense of Bloch and Kriz \cite{BKMTM}. 

The construction of the family of cycles provides in fact a differential system
for these cycles related to the action of the free Lie algebra $\Lie(X_0,X_1)$ on
itself by Ihara's special derivations. In this subsection we will associated
bar elements to the previously defined algebraic cycles using the unit of the
bar/cobar adjunction in the commutative algebra/Lie coalgebra case.

Before dealing with the algebraic cycles situations, we need to recall the
combinatorial situation from \cite{SouMPCC} and its relation with Ihara action. 
This is need to the related the Lie coalgebra situation (dual to Ihara action)
with  the differential system for algebraic cycles.


\subsection{Lie algebra, special derivations and Lyndon words}\label{subsec:Lie}

We present here the Lyndon brackets basis for the free Lie algebra
$\Lie(X_0,X_1)$ and its action on itself by special derivations. The Lie bracket
of  $\Lie(X_0,X_1)$ is denoted by $[\, ,\,]$  as usual.

A Lyndon word in $0$ and $1$ is a word in $0$ and $1$ strictly smaller than any
of its nonempty proper right factors for the lexicographic order with $0<1$ (for
more details, see \cite{ReuFLA93}). 

The standard factorization $[W]$ of a Lyndon word $W$ is defined inductively by
$[0]=X_0$, $[1]=X_1$ and otherwise by  $[W]=[[U],[V]]$ with $W=UV$, $U$ and $V$ nontrivial
and such that $V$ is minimal. 
\begin{exm}Lyndon words in letters $0<1$ in lexicographic order are up to weight
  $4$: 
\[
0<0001<001<0011 <01<011<0111<1
\]
Their standard factorization is given in weight $1$ $2$ and $3$ by
\[
[0]=X_0, \quad [1]=X_1, \quad [01]=[X_0,X_1], \quad [001]=[X_0,[X_0,X_1]] 
\mx{and} [[X_0,X_1],X_1].
\]
In weight $4$, one has
\[
[0001]=[X_0,[X_0,[X_0,X_1]]],
\quad \mx{and } \quad
[0111]=[[[X_0,X_1],X_1],X_1]
\]
and 
\[
[0011]=[X_0,[[X_0,X_1],X_1]]
\]
\end{exm}

The sets of Lyndon brackets $\{[W]\}$, that is Lyndon
words in standard factorization, form a basis of $\Lie(X_0,X_1)$ (\cite[Theorem
5.1]{ReuFLA93}). This basis  can then 
be used to write the Lie bracket:
\begin{defn}\label{defn:alpha} 
For any Lyndon word $W$, the coefficients $\alpha_{U,V}^W$ (with
  $U<V$ Lyndon words) are defined by:  
\[
[[U],[V]]=
\sum_{\substack{W\, \mx{{\scriptsize Lyndon}} \\ \mx{{\scriptsize words}} }}
\alpha_{U,V}^W[W].
\]
with $U<V$ Lyndon words. 
The $\alpha$'s are the structure coefficients of $\Lie(X_0,X_1)$.
\end{defn}

A derivation of $\Lie(X_0,X_1)$ is a linear endomorphism satisfying
\[
D([f,g])=[D(f),g]+[f,D(g)] \qquad \forall f,g \in \Lie(X_0,X_1). 
\]
\begin{defn}[Special derivation, \cite{IharaAPBG,IharaSDABG}] For any $f$ in
  $\Lie(X_0,X_1)$ we define a derivation $D_f$ by:
\[
D_f(X_0)=0, \qquad D_f(X_1)=[X_1,f].
\]
Ihara bracket on $\Lie(X_0,X_1)$ is given by
\[
\{f,g\}=[f,g]+D_f(g)-D_g(f).
\]
\end{defn}
Ihara bracket is simply the bracket of derivation 
\[
[D_1,D_2]_{Der}=D_1\circ D_2 - D_2 \circ D_1
\]
restricted to special derivations : 
\[
[D_f,D_g]_{Der}=D_h, \qquad \mx{with } h=\{f,g\}.
\]


Let $\TL[1]$ and $\TL[x]$ be two copies of the vector space $\Lie(X_0,X_1)$. The
subscript $x$ denotes a formal variable but it can be think as a point $x$ in
$\A^1$. $\TL[x]$ is endowed with the free bracket $[\, , \,]$ of $\Lie[X_0,X_1]$
while $\TL[1]$ is endowed with Ihara bracket $\{\,,\,\}$. The Lie algebra $\TL[1]$
acts on $\TL[x]$ by special derivations ; which act on $X_1$ hence the
subscript. 
If $f$ is an element of $\Lie(X_0,X_1)$, we write $f(1)$ its image in $\TL[1]$
and $f(x)$ its image in $\TL[x]$.

\begin{defn}[{\cite{GOVFLTLTG}}]
The semi-direct sum $\TL[1;x]$ of $L_x$ by $L_1$ is as a vector spaces the direct
sum
\[
\TL[1;x]=\TL[x]\oplus \TL[1]
\]
with bracket $\{\,,\,\}_{1;x}$ given by $[\, , \,]$ on $\TL[x]$, by $\{\,,\,\}$
on $\TL[1]$ and by 
\[
\{g(1),f(x)\}_{1;x}=-\{f(x),g(1)\}_{1;x}=D_g(f)(x) \qquad \forall f,g \in \Lie(X_0,X_1) 
\] 
on cross-terms.
\end{defn}

The union of $\{[W](x), [W](1)\}$ for all Lyndon words gives a basis of
$\TL[1;x]$ while a   basis of $\TL[1;x] \w \TL[1;x]$ is given by the union of
the following families
\begin{align*}
[U](x)\w[V](x) & \mx{ for any Lyndon word } U<V\\
[U](x)\w[V](1)  & \mx{ for any Lyndon word } U\neq V\\
[U](1)\w[V](1) & \mx{ for any Lyndon word } U<V. 
\end{align*}

\begin{defn}The structure coefficients $\alpha_{U,V}^W$,
  $\beta_{U,V}^W$ and $\gamma_{U,V}^W$ of
  $\TL[1;x]$ are given for any Lyndon words $W$ by the
  family of relations
\begin{equation}
\begin{aligned}\label{eqdef:alphabeta}
  \{[{U}](x) ,[ {V}](x) \}_{1;x}&=\sum_{W \in
  Lyn}\alpha_{U,V}^W [ {W}](x) . & \mx{for any Lyndon word } U<V\\
\{[{U}](x) ,[ {V}](1) \}_{1;x}&=\sum_{W \in
  Lyn}\beta_{U,V}^W [ {W}](x)   & \mx{for any Lyndon word } U\neq V\\
\{[ {U}](1) ,[ {V}](1) \}_{1;x}&=\sum_{W \in
  Lyn}\gamma_{U,V}^W [ {W}](1)  & \mx{for any Lyndon word } U<V. 
\end{aligned}
\end{equation}
 All coefficients above are integers.
\end{defn}
Because $\{\,,\,\}_{1;x}$ restricted to $\TL[x]$ is the usual cobracket on
$\Lie(X_0,X_1)$, the $\alpha_{U,V}^W $ are the $\alpha$'s of Definition
\ref{defn:alpha} ; Similarly the $\gamma$'s are the structure coefficients of
Ihara bracket.

Special derivations acts on $X_1$ and $D_{X_0}$ is simply bracketing
with $X_0$. This and the above remark show:
\begin{lem}[{\cite[Lemma 4.18]{SouMPCC}}]\label{relV01} 
Let $W$ be a Lyndon word of length greater than or equal to
  $2$.
Then the following holds for any Lyndon words $U,V$ :
\begin{itemize}
\item $\beta_{0,V}^W=0$,
\item $\beta_{V,0}^W=\alpha_{0,V}^W$
\item $\beta_{U,1}=0$,
\item $\beta_{1,U}^W=\alpha_{U,1}^W$.
\item $\gamma_{U,V}^W=\alpha_{U,V}^W+\beta_{U,V}^W-\beta_{V,U}^W$.
\end{itemize}
In particular, $\beta_{0,0}^W=\beta_{1,1}^W=0$. We also have 
\[
\alpha_{U,V}^{\ve}=\beta_{U,V}^{\ve}=\gamma_{U,V}^{\ve}=0
\]
for $\ve \in\{0,1\}$.
\end{lem}

\subsection[Coaction and Lie coalgebra]{The dual setting : a coaction and a Lie
  coalgebra }\label{subsec:coLie} 
The Lie algebra $Lie(X_0,X_1)$ is graded by the number of letters appearing inside a
bracket. Hence there is an induced grading on $\TL[1;x]$. Taking the graded dual
of the $\TL[1;x]$, we obtain a Lie coalgebra $\Tcl[1;x]$.

\begin{defn}
The elements of the  dual basis of the Lyndon bracket basis $[W](x)$ of $\TL[x]$
are denoted by $\T W(x)$. Similarly, $\T W(1)$ denotes, for a Lyndon Word $W$
the corresponding element in the basis dual to the basis of $\TL[1]$ given by the
$[W](1)$'s.  
\end{defn}
For $a$ in $\{1,x\}$, the elements $\T W(a)$ can be represented by a linear
combination of rooted trivalent tree with leaves decorated by $0$ and $1$ and
root decorated by $a$ (cf. \cite[Section 4.3]{SouMPCC}). This remark explains
the notation which is the same as in \cite{SouMPCC}.

A basis of $\Tcl[1;x]\w \Tcl[1;x]$ is given by the union of the
following families:
\begin{align*}
\T U(x)\w\T V(x) & \mx{ for any Lyndon word } U<V\\
\T U(x)\w\T V(1)  & \mx{ for any Lyndon word } U\neq V\\
\T U(1)\w\T V(1) & \mx{ for any Lyndon word } U<V. 
\end{align*}

By duality between $\TL[1;x]$ and $\Tcl[1;x]$ one has the following:
\begin{prop} \label{dcyTw}The bracket $\{,\}_{1;x}$ on $\TL[1,x]$ induces a
  cobracket on $\Tcl[1;x]$  
\[
\dc : \Tcl[1;x] \lra \Tcl[1;x] \w \Tcl[1;x]. 
\]
In terms of the above basis one gets 
\begin{equation} \tag{ED-T}\label{ED-T}
\dc(\T W(x))=\sum_{U<V} \alpha_{U,V}^W\T U(x) \w \T V(x) + 
\sum_{U,V} \beta_{U,V}^W\T U(x) \w \T V (1)
\end{equation}
and 
\begin{equation}\label{eq:dcyTw1}
\dc(\T W(1))=\sum_{U<V} \gamma_{U,V}^W\T U(1) \w \T V(1) 
\end{equation}
where $U$ and $V$ are Lyndon words.  The coefficients $\alpha_{U,V}^W$, $\beta_{U,V}^W$ and
$\gamma_{U,V}^W$ are those defined in Equation \eqref{eqdef:alphabeta}.
\end{prop}

Note that one has 
\[
\dc(\T 0(x))=\dc(\T 0(1))=0 \qquad \mx{and} \qquad 
\dc(\T 1(x))=\dc(\T 1(1))=0.
\]
In weight $2$ one has 
\[
\dc(T_{01}(x))=\T 0 (x)\w \T 1(x)+ \T 1(x)\w \T 0(1). 
\]

Because of geometric constraints, one can not use directly the combinatorics of
the cobracket $\dc$ in this basis to defined a family of algebraic cycle over
$\ps$. 
One defines for any Lyndon word $W$
\[
\Tc W = \T W(x) \qquad \mx{and} \qquad \Tcu W=\T W(x)-\T W(1).
\]
The definition of $\Tc W$ can be thought as the difference $\Tc W=\T W(x)-\T
W(0)$ where the element $\T W(0)$ is equal to $0$. The elements $\Tc W$ and
$\Tcu W$ form a basis of $\Tcl[1;x]$ when $W$ runs through the set of Lyndon
words.

\begin{lem}[{\cite[Lemma 4.32]{SouMPCC}}]\label{lem:aapbbpa}
In this basis the cobracket $\dc$ is given by
\begin{equation}\label{ED-Tc}
\dc(\Tc W)=\sum_{U<V} a_{U,V}^W\Tc U \w \Tc V +
 \sum_{U,V}b_{U,V}^W\Tcu{U}\w \Tc V, 
\tag{ED-$\Tcg^0$}
\end{equation}
and 
\begin{equation}\label{ED-Tcu}
\dc(\Tcu W)=\sum_{U<V} \ap_{U,V}^W\Tcu U \w \Tcu V +
\sum_{U,V}\bp_{U,V}^W\Tcu{U}\w\Tc V
\tag{ED-$\Tcg^1$}
\end{equation}
where the coefficients $a$'s, $b$'s $\ap$'s and $\bp$'s are given by
\begin{equation}\label{abcoef}
\begin{array}{ll}
a_{U,V}^W=\alpha_{U,V}^W+\beta_{U,V}^W -\beta_{V,U}^{W} & \mx{ for } U<V\\[0.5em]
b_{U,V}^W=\beta_{V,U}^W& \mx{ for any }U,\, V
\end{array}
\end{equation}
and
\begin{equation}\label{apbpcoef}
\begin{array}{ll}
\ap_{U,V}^W=-a_{U,V}^W&\mx{ for } U<V,\\[0.5em]
\bp_{U,V}^W=a_{U,V}^W+b_{U,V}^W&\mx{ for } U<V,\\[0.5em]
\bp_{V,U}^W=-a_{U,V}^W+b_{V,U}^W&\mx{ for } U<V,\\[0.5em]
\bp_{U,U}^W=b_{U,U}^W&\mx{ for any } U.
\end{array}
\end{equation}
\end{lem} 

From the explicit description of the coaction,  Lemma
\ref{relV01} gives explicitly some of the coefficients $\alpha$'s and
$\beta$'s. This translates as

\begin{lem}[{\cite[Lemma 4.33]{SouMPCC}}]\label{rem:coefAL}
\begin{itemize}~
\item If $W$ is the Lyndon word $0$ or $1$, then :
\[
a_{U,V}^0=b_{U,V}^0=\ap_{U,V}^0=\bp_{U,V}^0=0,
\quad
a_{U,V}^1=b_{U,V}^1=\ap_{U,V}^1=\bp_{U,V}^1=0
\]
for any Lyndon words $U$ and $V$.
\item  For any Lyndon word $W$, $U$ and $V$ of length at least $2$, one has
\[
a_{0,V}^W=\ap_{0,V}^W=0, 
\quad 
\bp_{U,0}^W=\bp_{U,0}^W= 0
\] 
and 
\[
a_{U,1}^W=\ap_{U,1}^W=0, 
\quad 
b_{1,V}^W=\bp_{1,V}^W=0.  \]
\end{itemize}
\end{lem}
Moreover for  $W$ a Lyndon word, 
\[
a_{U,V}^W=b_{U,V}^W=\ap_{U,V}^W=\bp_{U,V}^W=0
\]
as soon as the length of $U$ plus the length of  $V$ is not equal to the length of $W$.

From the definition of $a_{U,V}^W$ and Lemma \ref{relV01} one sees that
$a_{U,V}^W=\gamma_{U,V}^W$. This and Equations \eqref{ED-Tc} and \eqref{ED-Tcu}
shows that  
\begin{align}
\dc(\T W(1))=\dc(\Tc W -\Tcu W)=&
\sum_{U<V} a_{U,V}^W
(\Tc U - \Tcu U)\w (\Tc V - \Tcu V)\label{eq:dcyTcTcu}\\
=&\sum_{U<V} a_{U,V}^W
(\Tc U(1))\w (\Tc V(1))\label{eq:dcyTc(1)}
\end{align}

\subsection{A differential system for
  cycles}\label{subsec:diffsysforcycle}
In this subsection we review the cubical complex of quasi-finite cycles over $X$
computing higher Chow groups of $X$. This complex has a natural cdga structure.
M. Levine, in  \cite{LEVTMFG}, proved that  the $\HH^0$ of its bar construction
is the tannakian Hopf algebra of mixed Tate motive over $X$.

The ground field is $\Spec(\Q)$. The projective line minus three points $\ps{}$
will be simply denoted by $X$. A generic smooth quasi-projective variety will be
denoted by $Y$.

We define $\square^1$ to be $\square^1=\p^1 \sm \{1\}$ and $\square^n$ to be
$(\square^1)^n$. The standard projective coordinates on $\square^n$ is
$[U_i:V_i]$ on the $i$-th factor; and $u_i=U_i/V_i$ is the corresponding affine
coordinate. A face $F$ of codimension $p$  of $\square^n$ is given by
$u_{i_k}=\ve_k$ for $k=1,\ldots , p$ and $\ve_k$ in $\{0,\infty\}$. Such a face
is isomorphic to $\square^{n-p}$. For $\ve=0,\infty$ and $i$ in $\{1,\ldots, n\}$,
let $s_{i}^{\ve}$ denote the insertion 
morphism of a codimension $1$ face 
\[
s_i^{\ve} : \square^{n-1} \lra \square^n
\] 
given by the identification
\[
\square^{n-1} \simeq \square^{i-1} \times \{ \ve\} \times \square^{n-i}.
\]

\begin{defn}[{\cite[Example 4.1.6]{LEVTMFG}}]
\label{def:qfcycle}Let $Y$ be an irreducible smooth variety.
\begin{itemize}
\item Let $\Zqf[p](Y,n)$ denote the free abelian group
  generated by irreducible closed subvarieties 
\[Z \subset Y\times \square^n \times(\p^1\sm\{1\})^p
\] 
 such that the  restriction of the projection on $Y \times \square^n$,
\[
 p_1 : Z \lra Y\times \square^n,
\]
is dominant and quasi finite (that is of pure relative dimension  $0$).
\item We say that elements of $\Zqf[p](Y,n)$ are \emph{quasi-finite}.
\end{itemize}
\end{defn}
Intersection with codimension $1$ faces give morphisms
\[
\dN_{i}^{\ve}={(s_i^{\ve})}^* : \Zqf[p](Y,n) \lra \Zqf[p](Y,n-1)
\]

The symmetric group $\Sn[p]$ acts on $\Zqf[p](Y,n)$ by permutation of the
  factors of $(\p^1\sm\{1\})^p$. Let $Sym_{\p^1\sm\{1\}}^p$ denotes the projector
  corresponding to the \emph{symmetric} representation.

The symmetric group $\Sn[n]$ acts on $\Zqf[p](Y,n)$ by permutation of the factor
$\square^1$, and $(\Z/2\Z)^n$ acts on $\Zqf[p](Y,n)$ by $u_i \mapsto 1/u_i$ on
the $\square^1$. The sign representation of $\Sn[n]$ extends to a sign
representation  
\[
G_n=(\Z/2\Z)^n\rtimes\Sn[n] \lra \{1,-1 \}.
\] 
Let $\Alt_n\in \Q[G_n]$ be the corresponding projector.

\begin{defn} Let $\Ngqf{Y}{k}(p)$ denote 
\[
\Ngqf{Y}{k}(p)=Sym_{\p^1\sm\{1\}}^p\circ \Alt_{2p-k}\left(\Zqf[p](Y,2p-k )\otimes \Q\right).
\]

\begin{itemize}
\item The intersection with codimension $1$ faces of
  $\square^{2p-k}$ induces a differential
\[
\da[Y]=\sum_{i=1}^{2p-k} (-1)^{i-1} (\dN_{i}^0 - \dN_{i}^{\infty}) 
\]
of degree $1$.
\item The complex of quasi finite cycles is defined by 
\[
\Ngqf{Y}{\bullet}=\Q \oplus \bigoplus_{p \geqs 1}\Ngqf{Y}{\bullet}(p).
\]
\item  Concatenation of factors $\square^n$ and of
  factors $ \p^1\sm\{1\}$ followed by the  pull-back by the
diagonal  gives a product structure to  $\Ngqf{X}{\bullet}$. This product is
graded commutative and $\Ngqf{X}{\bullet}$ is a cdga (\cite[Section 4.2]{LEVTMFG}). 
\end{itemize}
\end{defn}

Thanks to \cite[Chapter IV and VI]{FSVCtrMHT}, the cohomology of
$\Ngqf{X}{\bullet}$ agrees with 
higher Chow groups of $Y$ tensored with $\Q$  (one can also see \cite[Lemma
4.2.1]{LEVTMFG}).
  
In \cite{SouMPCC}, the author defined two weight $1$ degree $1$ 
cycles $\Lco$ and $\Lcz$ in $\cNg{X}{1}$ as the image under
$Sym_{\p^1\sm\{1\}}^1\circ \Alt_{1}$  of the irreducible varieties defined 
respectively by: 
\[
Z_0 \subset X\times \square^1\times (\p^1\sm\{1\}): (U-V)(A-B)(U-xV)+x(1-x)UVB=0
\] 
and
\[
Z_1 : (U-V)(A-B)(U-(1-x)V)+x(1-x)UVB=0.
\] 

Starting with these two cycles, the author built in \cite{SouMPCC} two families
of degree $1$ elements in $\Nge{X}{\bullet}$ whose differential are given by the
cobracket in $\Tcl[1;x]$.

Let $j$ be the inclusion $\ps=X  \hookrightarrow   \A^1$. The differential on
$\Nge{X}{\bullet}$ is simply denoted by $\dN$.
\begin{thm}[\cite{SouMPCC}]\label{thm:cycleLcLcu}  For any Lyndon word of length
  $p\geqs2$, there exist
  two cycles $\Lc_W$ and $\Lcu_W$ in $\Nge{X}{1} (p)$ such that:
\begin{itemize}
\item There exist cycles $\ol{\Lc_W}$, $\ol{\Lcu_W}$ in $\Nge{\A^1} 1 (p)$ such
  that
\[
\Lc_W=j^*(\ol{\Lc_W}) \qquad \mx{and} \qquad \Lcu_W=j^*(\ol{\Lcu_W}). 
\]
\item The restriction of $\ol{\Lc_W}$ (resp. $\ol{\Lcu}$) to the fiber $t=0$ (resp. $t=1$)
  is empty. 
\item The cycle $\Lc_W$ and $\Lcu_W$ satisfy the following differential
  equations in $\Nc$:
\begin{equation}\label{ED-Lc}
\dN(\Lc_W)=-\left(\sum_{U<V} a_{U,V}^W\Lc_U \Lc_V + \sum_{U,V}b_{U,V}^W\Lcu_{U}\Lc_V\right)
\tag{ED-$\Lc$}
\end{equation}
and 
\begin{equation}\label{ED-Lcu}
\dN(\Lcu_W)=-\left(\sum_{U<V} \ap_{U,V}^W\Lcu_U \Lcu_V +
\sum_{U,V}\bp_{U,V}^W\Lcu_{U}\Lc_V\right)
\tag{ED-$\Lcu$}
\end{equation}
\end{itemize}  
where coefficients $a$'s, $b$'s, $\ap$'s and $\bp$'s are the ones of equations
\eqref{ED-Tc} and \eqref{ED-Tc}.
\end{thm}

One will write generically a cycle in the above families as $\Lce_W$ with
$\ve$ in $\{0, 1\}$ when working over $X=\ps$ and $\ol{\Lce_W}$ when working
over $\A^1$.

The above equations differ from the cobracket in $\Tcl[1;x]$ given at equations
\eqref{ED-Tc} and \eqref{ED-Tcu} by a global minus sign. This is due to a shift
in the degree.  Hence the above cycles $\Lce_W$ differ from the  ones defined
in \cite{SouMPCC} by a global minus sign. 

\begin{rem}\label{rem:A1eqdiff}
The extension   $\ol{\Lce_W}$ of the cycles to $\Nge{\A^1}{\bullet}$
 satisfy the same differential equations as $\Lce_W$  by
  considering the Zariski closure of the   product in the R.H.S. of
  \eqref{ED-Lc} and \eqref{ED-Lcu}. However, this Zariski closure is \emph{not}
  decomposable: terms of the form $\ol{\Lcz \Lc_V}$ are not product in
  $\Nge{\A^1}{\bullet}$ because $\ol{\Lcz}$ is not in $\Nge{\A^1}{1}$ ; it is
  not quasi-finite over $0$ (cf proof of Proposition 6.3 in \cite{SouMPCC}). 
\end{rem}

Despite the above remark, Theorem  \ref{thm:cycleLcLcu} and the proof of Theorem
5.8 in \cite{SouMPCC} give two other but related families of cycles with
decomposable boundary in $\Nge{\A^1}{1}$. Their are described below. 

Let $W$ be a Lyndon word of length greater than $2$. We define $\ol{\LcLcu_W}$
to be the difference
\[
\ol{\LcLcu_W}=\ol{\Lc_W} -\ol{\Lcu_W}. 
\]
The geometric situation relates $\ps$, $\A^1$ and the point $\{1\}$ as follows:
\[
\begin{tikzpicture}
\matrix (m) [matrix of math nodes,
 row sep=2em, column sep=5em]
{X=\ps & \A^1\\
  &\{1\} \\
};
\path[->,font=\scriptsize]
(m-1-1) edge node[auto] {$j$} (m-1-2)
(m-1-2) edge node[left] {$p_1$} (m-2-2);
\path[->,bend right]
(m-2-2) edge node[right] {$i_1$} (m-1-2);
\end{tikzpicture}
\]
where $j$ is the open inclusion, $p_1$ is the projection onto $\{1\}$ and $i_1$
the closed inclusion (or the $1$-section). We define the constant cycle
$\ol{\Lc_W(1)}$ as 
\[
\ol{\Lc_W(1)}=p_1^*\circ i_1^* (\ol{\Lc_W})=p_1^*(\ol{\Lc_W}|_{x=1})
\]
where $\ol{\Lc_W}|_{x=1}$ denotes the fiber at $1$ of the cycle
$\ol{\Lc_W}$. Its restriction to $X$ is denote by $\Lc_W(1)$.
\begin{lem}\label{rem:XA1LcLcu} 
For any Lyndon word of length $p\geqs 2$ the cycle $\ol{\LcLcu_W}$ satisfies
\begin{equation}
\label{ED-LcLcu}
\dN\left(\ol{\Lc_W}-\ol{\Lcu_W}\right)=-\left(\sum_{0<U<V<1} a_{U,V}^W
\left(\ol{\Lc_U}-\ol{\Lcu_U}\right)\left(\ol{\Lc_V}-\ol{\Lcu_V}\right)
\right). 
\tag{ED-$\ol{\mathcal L^{0-1}}$}
\end{equation}

The differential of $\ol{\Lc_W(1)}$ is given by 
\begin{equation}\label{ED-Lc1}
\dN\left(\ol{\Lc_W(1)}\right)=
-\left(\sum_{0<U<V<1} a_{U,V}^W\ol{\Lc_U(1)}\ol{\Lc_V(1)} \right)
\end{equation}
The above equation also holds for $\Lc_W(1)$ and $i_1^*(\ol{\Lc_W})$.
\end{lem}
\begin{proof}
The combinatoric being the same as in $\Tcl[1;x]$, the first part follows from
Equation \eqref{eq:dcyTcTcu}. The second part is a consequence of Lemma
\ref{rem:coefAL} because product of the form $\ol{\Lcu_U \Lc_V}$ have empty
fiber at $1$.
\end{proof}

The rest of this section shows that each family of ``differential system'' gives
rise to a family of elements in the corresponding bar constructions.

Let $B_X$, $B_{\A^1}$ and $B_{\{1\}}$ denote the bar
  construction over 
  $\Nge{X}{\bullet}$, $\Nge{\A^1}{\bullet}$ and $\Nge{\Sp(\Q)}{\bullet}$
  respectively. Let $Q_X$, 
  $Q_{\A^1}$ and $Q_{\{1\}}$ be the corresponding set of indecomposable
  elements. By an abuse of  notation, we will write $d_B$, $\Delta$, $\sha$ and
  $\delta_Q$ the natural operation in the corresponding spaces. When required by
  the context,
  we will precise the ``base'' space by the subscript $X$, $\A^1$ and $\{1\}$
  respectively. 

Note that the geometric relation between $X=\ps$, $A^1$ and $\{1\}$ gives rise
to morphisms of cdga between the corresponding cycles algebras: 
\[
\begin{tikzpicture}
\matrix (m) [matrix of math nodes,
 row sep=2em, column sep=5em]
{
\Nge{X}{\bullet} & \Nge{\A^1}{\bullet} \\
  &\Nge{\Sp(\Q)}{\bullet} \\
};
\clip[->,font=\scriptsize]
(m-1-2) edge node[auto] {$j^*$} (m-1-1)
(m-2-2) edge node[left] {$p_1^*$} (m-1-2);
\path[->,bend left]
(m-1-2) edge node[right] {$i_1^*$} (m-2-2);
\end{tikzpicture}
\]
which then induce morphisms between bar construction and sets of indecomposable
elements. These morphisms are also denoted $j^*$, $p_1^*$ and $i_1^*$.

  
\begin{thm}[bar elements]\label{bareltX}
For any Lyndon word $W$ of length $p$ there exist an element $\Lcb_W$,
in the bar construction  $B_X$ satisfying:
\begin{itemize}
\item Its image under the projection onto the tensor degree $1$ part $\pi_1 :
  B_X \lra \cNg{X}{\bullet}$ is $\pi_1(\Lcb_W)=\Lc_W$.
\item Its is in the image of the projector $\psh$; hence it is in $Q_X$.
\item It is  of bar degree $0$ and its image under $d_B$ is $0$. Thus it induced
  a class in $\HH^0(B_X)$ and in $\HH^0(Q_X)=Q_{\HH^0(B_X)}$.
\item Its image under $\delta_Q$ is given by the
  differential equations \eqref{ED-Lc} without the minus sign
\[
\delta_{Q}(\Lc_W)=\sum_{U<V} a_{U,V}^W\Lc_U \Lc_V +
  \sum_{U,V}b_{U,V}^W\Lcu_{U}\Lc_V.
\]
\end{itemize}
A similar statement holds for $\Lcu_W$, $\LcLcu_W$  and $\Lc_W(1)$ with replace
Equation \eqref{ED-Lc}
by equation \eqref{ED-Lcu}, \eqref{ED-LcLcu} and \eqref{ED-Lc1}. 
\end{thm}
\begin{proof}
The main point is the relation between $\Tcl[1;x]$ and the above family of
algebraic cycles and to use the unit of the adjunction cobar/bar.

As in Section \ref{subs:cobarcoLie},
$\Omega_{coL}(\Tcl[1;x])=S^{gr}(s\Tcl[1;x])$ denotes the 
cobar construction over the Lie coalgebra $\Tcl[1;x]$ concentrated purely in
degree $0$ (hence with $0$ as differential).   Let 
\[
\psi : \Omega_{coL}(\Tcl[1;x]) \lra \Nge{X}{\bullet}
\]
be the morphism of cdga induced by 
\[
s\Tc W \longmapsto \Lc_W, \qquad s\Tcu W \longmapsto \Lcu_W
\]
for any Lyndon word $W$ of length $p\geqs 2$ together with 
\[
\psi(s\Tcu 0)=\Lco,\quad \psi(s\Tc 1)=\Lcz \mx{ and }
\psi(s\Tc 0)=\psi(s\Tcu 1)=0
\]
where the prefix $s$ denotes the suspension. 

The morphism $\psi$ is compatible with the differential because of equations \eqref{ED-Lc}
and\eqref{ED-Lcu} and Lemma \ref{rem:coefAL}. The minus sign difference between
equations \eqref{ED-Lc} and \eqref{ED-Tc} (and similarly for Equation
\eqref{ED-Lcu}) makes it possible to define $\psi$ without sign (cf Section
\ref{subs:cobarcoLie}) 

It induces a morphism on the bar construction (for the associative case)
\[
\psi_B : B(\Omega_{coL}(\Tcl[1;x])) \lra B_X=B(\Nge{X}{\bullet})
\]
compatible with projection on tensor degree $n$ part (for any $n$) and with the
projector $\psh$ onto the indecomposable elements. 

Hence we obtain the following commutative diagram (of vector space) 
\[
\begin{tikzpicture}
\matrix (m) [matrix of math nodes,
 row sep=3em, column sep=7em]
{ & \Omega_{coL}(\Tcl[1;x]) & \Nge{X}{\bullet}  \\
 &B(\Omega_{coL}(\Tcl[1;x])) & B_X\\
\Tcl[1;x] & B_{com}(\Omega_{coL}(\cL))& Q_X \\
};
\path[->,font=\scriptsize]
(m-3-1) edge node[auto] {$\upsilon$} (m-3-2.west)
(m-3-1) edge node[below,xshift=-0.5ex] {} (m-1-2)
(m-1-2) edge node[auto] {$\psi$} (m-1-3)
(m-2-2) edge node[auto] {$\psi_{B}$} (m-2-3)
(m-3-2) edge node[auto] {$\psi_{Q}$} (m-3-3)
(m-2-2) edge node[auto] {$\pi_1$} (m-1-2)
(m-2-3) edge node[auto] {$\pi_1$} (m-1-3)
(m-3-2) edge node[auto] {} (m-2-2)
(m-3-3) edge node[auto] {} (m-2-3);
\end{tikzpicture}
\]
where the morphisms in the bottom line are morphisms of dg Lie algebras and
where $\upsilon$ is the unit of the bar/cobar adjunction.

The bar element $\Lcb_W$ is then defined by
\[
\Lcb_W=\psi_Q\circ \upsilon(\Tc W).
\]
Similarly we define $\Lcub_W$ and $\LcLcub W$.

In order to defined $\Lcb_W(1)$, one consider only the sub Lie coalgebra
$\Tcl[1]$ of $\Tcl[1;x]$ and the morphism 
\[
\psi : s\T W(1) \longmapsto \Lc_W(1)
\]
when $W$ has length $p \geqs 2$ and sending $s\T 0(1)$ and $s\T 1(1)$ to zero.
\end{proof}

Over $\A^1$ a similar statement holds:
\begin{prop}\label{bareltA1}
 For any Lyndon word $W$ of length $p\geqs 2$, there exists an element
 ${\LcLcub W}$ in the bar construction $B_{\A^1}$ satisfying 
\begin{itemize}
\item $\pi_1(\LcLcub W)=\ol{\LcLcu_W}=\ol{\Lc_W} -\ol{\Lcu_W}$.
\item It is  in the image of the projector
$\psh$; hence in $Q_{\A^1}$.
\item It is of bar degree $0$ and their image under $d_B$ is $0$; hence it
  give a  class in $\HH^0(B_X)$ and in $\HH^0(Q_X)$.
\item Its image under $\delta_Q$ is given by Equation \eqref{ED-LcLcu} without
  the   minus sign. 
\item $j^*(\LcLcub W)=\Lcb_W - \Lcub_W$.
\end{itemize}

A similar statement holds for $\ol{\Lc_W(1)}$ and
$i_1^*(\ol{\Lc_W})=i_1^*(\ol{\Lc_W(1)})$ with 
\eqref{ED-LcLcu} replace by \eqref{ED-Lc1}. The corresponding bar elements are
denoted 
\[
\ol{\Lcb_W(1)}
 \qquad \mx{and}\qquad \Lcb_{W,x=1}
\]
respectively.
We have  the appropriate compatibilities with $p_1^*$ and $j^*$.

\end{prop} 
\begin{proof}
The proof goes as in Theorem \ref{bareltX} above but using only the sub Lie
coalgebra $\Tcl[1]$ and equations \eqref{ED-LcLcu} and \eqref{ED-Lc1}.
The coefficients $a_{U,V}^W$ appearing in the differential equation for the cycle
are equal to coefficients $\gamma_{U,V}^W$ giving the cobracket of the element $\T
W (1)$ (cf.  Lemma \ref{relV01}).

The relations between bar elements over  $\{1\}$, $\A^1$ and $X=\ps$ follow
because $i_1^*$, $p_1^*$ and $j^*$ are morphisms of cdga algebra.
\end{proof}

\begin{lem}\label{lemrestA1=0}
In $B_{\Q}$ the following relation holds :
\[
i_1^*(\LcLcub W)=i_1^*(\ol{\Lcb_W(1)})=\Lcb_{W,x=1}.
\]
\end{lem}
\begin{proof}
It follows from equations \eqref{eq:dcyTcTcu} and \eqref{eq:dcyTc(1)} which
holds for the cycle on $\A^1$ and because in the cycle setting one has in
$\Nge{\Q}{\bullet}$: 
\[
i_1^*(\LcLcu_W)=i_1^*(\ol{\Lc_W(1)})=\Lcb_{W}|_{x=1}
\]
for any Lyndon word $W$.
\end{proof}

\section{A relative basis for mixed Tate motive over $\ps$}
\subsection{Relations between bar elements}\label{sec:relbar}
A motivation for introducing cycles $\Lcu_W$ in \cite{SouMPCC} was the  idea of
a correspondence 
\[
\Lc_W - \Lc_W(1) \leftrightarrow \Lcu_W
\]
In this section, we prove that this relation is 
 an equality in the
$\HH^0$ of the bar construction modulo shuffle products; that is 
\[
\Lcb_W - \Lcb_W(1) = \Lcub_W \qquad \in \HH^{0}(Q_X).
\]


A key point in order to build
cycles $\Lc_W$ and $\Lcu_W$ in \cite{SouMPCC} was a pull-back by the
multiplication. More precisely,  
the usual multiplication $ \A^1 \times \A^1 \lra \A^1$ composed with the
isomorphism $\A^1\times \square^1 \simeq \A^1 \times \A^1$  gives a
multiplication  
\[
m_0 : \A^1\times \square^1 \lra  \A^1
\]
sending $(t,u)$ to
$\frac{t}{1-u}$. Twisting $m_0$ by $\theta : t \mapsto
1-t$ gives a ``twisted multiplication''
\[
m_1 = (\theta \times \id)\circ m_0 \circ \theta : \A^1\times \square^1 \lra  \A^1.
\]

\begin{prop}[{\cite{SouMPCC}}] \label{mehomotopy}
For $\ve=0,1$ the
  morphism $m_{\ve}$ induces a linear morphism
\[
m_{\ve}^* : \Nge{\A^1}{k} \lra \Nge{\A^1}{k-1}
\]
giving a homotopy
\[
\da[\A^1] \circ m_{\ve}^* + m_{\ve}^* \circ \da[\A^1]
=\id - p_{\ve}^*\circ i_{\ve}^*
\]
where $p_{\ve} : \A^1 \lra \{\ve\}$ is the projection onto the point $\{\ve\}$ and
$i_{\ve}$ its inclusion is $\A^1$.
\end{prop}
From this homotopy property, one derives the following relation between $m_0^*$
and $m_1^*$. 
\begin{lem}\label{lem:m0->m1}
One has:
\begin{equation}\label{m1->m0}
m_1^*=m_0^*-p_1^*\circ i_1^*\circ m^*_0 -\da[\A^1]\circ m^*_1 \circ m_0^* + m^*_1 \circ
m_0^* \circ \da[\A^1]+m_1^*\circ p_0^*\circ i_0^*
\end{equation}
and a similar expression for $m_1^*$.

In particular, when $b \in \Nge{\A^1}{k}$ satisfies $\da[\A^1](b)=0$ and
$i_0(b)=0$, one has:
\[
m_1^*(b)=m_0^*(b)-p_1^*\circ i_1^*(m_0^*(b))+\da[\A^1]( m^*_0 \circ m_1^* (b))
\]
\end{lem}
\begin{proof}
Let $b$ be in $\Nge{\A^1}{k}$. We treat only Equation \ref{m1->m0}. Using the
homotopy property for $m_0^*$, one writes
\[
b=\da[\A^1]\circ m_0^*(b)+m_0^*\circ \da[\A^1] + p_0^* \circ i_0^*(b). 
\]
Computing $m_1^*(b)$, the homotopy property 
\[
m_1^*\circ \da[\A^1] (m_ 0(b))=m_0^*(b)-p_1^*\circ i_1^*(m_0(b))
-\da[\A^1]\circ m_ 1^*(m_0^*(b))
\]
gives the desired formula.
\end{proof}

Writing $\ol{A_W}$  the $\A^1$-extension of the  right hand
side of Equation \eqref{ED-Lc}, cycles $\ol{\Lc_W}$  are obtained in \cite{SouMPCC} as 
\[
\ol{\Lc_W}=m_0^*(\ol{A_W}) 
\]
and similarly for $\Lcu_W$. 

An explicit computation in low weight \cite{SouMPCC, SouPolex} shows that
\[
\ol{\Lc_{01}}=m_0^*(\ol{\Lc_0} \, \ol{\Lc_1})
\quad \mx{and} \quad
\ol{\Lcu_{01}}=m_1^*(\ol{\Lc_0} \, \ol{\Lc_1}).
\]

Using Lemma \ref{lem:m0->m1}, one gets
\begin{equation}\label{01-m0m1}
\ol{\Lcu_{01}}=\ol{\Lc_{01}}-\ol{\Lc_{01}(1)} +\da[\A^1]( m_ 1^*\circ
m_0^*(\ol{\Lc_0} \, \ol{\Lc_1})).
\end{equation}  

Note that as a parametrized cycle, 
\[
 m_ 1^*\circ m_0^*(\ol{\Lc_0} \, \ol{\Lc_1})=m_1^*(\ol{\Lc_{01}})
\]
 can be written (omitting the projector $\Alt$) as
\[
 m_ 1^*\circ m_0^*(\ol{\Lc_0} \, \ol{\Lc_1})=
\left[ t ;
\frac{y-t}{y-1}, 1- \frac y x ,x,1-x
\right] \subset \A^1 \times \square^4.
\] 
This expression coincides, up to reparametrization, with the expression of
$C_{01}$ given in \cite[Example 5.5]{SouPolex} relating $\Lc_{01}$ and
$\Lcu_{01}$.

Thus, $\ol{\Lc_{01}}-\ol{\Lc_{01}(1)}$ and $\ol{\Lcu_{01}}$ differs only by a boundary. The
differential of $\ol{\Lc_{01}(1)}$ is zero and one can compute explicitly the
corresponding bar elements:
\[
\LcLcub {01}=[\Lc_{01}]-[\Lcu_{01}] ,
\quad
\Lcb_{01}(1)=[\Lc_{01}(1)].
\]
\begin{lem}\label{Lc01Lcu01B}
In  $B(\Nge{A^1}{\bullet})$, one has the following relation 
\begin{equation}\label{O1B=A1}
\LcLcub {01}-\Lcb_{01}(1) = d_B([ m_ 1^*(\ol{\Lc_{01}})]).
\end{equation}
Thus in $\HH^0(B_{\A^1})$ as in $\HH^0(Q_{\A^1})$ one has the equality between
\[
\LcLcub {01}-\Lcb_{01}(1)=0.
\]
Taking the restriction to $\ps$, one obtains in $\HH^0(Q_X)$
\begin{equation}\label{O1B=X}
\Lcb_{01}-\Lcub_{01}=j^*(\LcLcub{01}) =\Lcb_{01}(1).
\end{equation}
\end{lem}

For $W$ a Lyndon word of length $p \geqs 2$, the explicit comparison between $\Lcb_W$
and $\Lcub_W$ is in general much more complicated as 
\[
\da[\A^1](\ol{\Lc_W}-\ol{\Lcu_W})=-\sum_{0<U<V<1}a^W_{U,V}
(\ol{\Lc_U}-\ol{\Lcu_U})(\ol{\Lc_V}-\ol{\Lcu_V}) \neq 0.
\]

However, working at the bar construction level in $\HH^0(Q_{\A^1})$ allows to
use an induction argument.
\begin{thm}\label{LcLcuBA1}
For any Lyndon word $W$ of length $p\geqs 2$ the following relation holds
\begin{equation}\label{eq:LcLcuBA1}
\LcLcub W =\ol{\Lcb_W(1)} \qquad \mx{in}\quad  \HH^0(Q_{\A^1})=Q_{\HH^0(B_{\A^1})}.
\end{equation}
Taking the restriction to $X=\ps$, one obtains in $\HH^0(Q_{X})=Q_{\HH^0(B_{X})}$
\begin{equation}\label{eq:LcLcuBX}
\Lcb_W-\Lcub_W=j^*(\LcLcub W) =\Lcb_W(1).
\end{equation}
\end{thm}
\begin{proof}
From  Lemma \ref{Lc01Lcu01B} above it is true for $p=2$ as there is then only
one Lyndon word to consider $W=01$.

Now we assume that the theorem is true for all Lyndon words of length $k$ with $2
\leqs k \leqs p-1$. Let $W$ be a Lyndon word of length $p$.

From Proposition \ref{bareltA1}, one has in $Q_{\A^1}$ and in particular in
$\HH^0(Q_{\A^1})$:
\[
\delta_Q(\LcLcub W)=\sum_{0<U<V<1} a^W_{U,V} (\LcLcub U) \w (\LcLcub V)
\]
and 
\[
\delta_Q(\ol{\Lcb_W(1)})=\sum_{0<U<V<1} a^W_{U,V} (\ol{\Lcb_U(1)}) \w (\ol{\Lcb_V(1)}).
\]
Using the induction hypothesis, one has in $\HH^0(Q_{\A^1})$
\[
\delta_Q(\LcLcub W)=\sum_{0<U<V<1} a^W_{U,V} (\ol{\Lcb_U(1)}) \w (\ol{\Lcb_V(1)})
\]
and thus 
\[
\delta_Q\left(\LcLcub W - \ol{\Lcb_W(1)} \right)=0 \qquad \mx{in} \quad
\HH^0(Q_{\A^1}). 
\]
Let $C_W$ be the class of $\LcLcub W - \ol{\Lcb_W(1)}$ in $\HH^0(Q_{\A^1})$ and
  $sC_W$ its image in 
\[
\Omega_{coL}(\HH^0(Q_{\A^1}))=S^{gr}(s\Q \otimes  \HH^0(Q_{\A^1}) ).
\]

As $\delta_Q(C_W)=0$, $d_{\Omega, coL}(sC_W)=0$ and $sC_W$ gives a class in 

\[
\HH^1(\Omega_{coL}(\HH^0(Q_{A^1})))\simeq \HH^1(\Nge{\A^1}{\bullet});
\]
where the above isomorphism is given by Bloch and Kriz in \cite[Corollary
2.31]{BKMTM} after a choice of a $1$-minimal model in the sens of
Sullivan. Using the comparison between $\HH^1(\cNg{\A^1}{\bullet})$ and the
higher Chow groups, this class can be represented 
by $p_1^*(C)$ in $\Nge{\A^1}{1}$ with $C$ a cycle in $\Nge{\Q}{1}$.

The cycle  $p_1^*(C)$ satisfies $\da[\A^1](p_1^*(C))=0$ and $[p_1^*(C)]$ gives a
degree $0$ bar element $C^B$ in $B_{\A^1}$ whose  bar differential and reduced coproduct
are equal to $0$.

From this, one gets a class $\tilde{C}_W=C_W-C^B$ in $\HH^0(Q_{\A^1})$. Its
image $s\tilde C_W$ in $\Omega_{coL}(\HH^0(Q_{A^1}))$ also gives a class in
\[
\HH^1(\Omega_{coL}(\HH^0(Q_{A^1})))
\]
which is $0$ by construction.

As, on the degree $0$ part of $\Omega_{coL}(\HH^0(Q_{\A^1}))$, the differential
$d_{\Omega, coL}$ is 
zero, one obtains that  $s\tilde C_W=0$ in $s\Q \otimes \HH^0(Q_{A^1})$ and thus
$\tilde C_W$ is zero in $\HH^0(Q_{A^1})=Q_{\HH^0(B_{\A^1})}$. The above
discussion shows that:
\[
0=\tilde C_W=C_W-C^B=C_W-[p_1^*(C)]
\]

So far one has obtained that in $B_{\A^1}$: 
\begin{equation}\label{eq:LcLcubW-C}
\LcLcub W - \ol{\Lcb_W(1)}-[p_1^*(C)]=d_B(b) \qquad \mx{modulo }\sha \mx{ products}
\end{equation}
with $b$ in the degree $-1$ part of $B_{\A^1}=B(\Nge{\A^1}{\bullet})$.

Because taking the fiber at $1$ commutes with products and differential,
one gets modulo shuffles
\[
i_1^*(\LcLcub W)- i_1^*(\ol{\Lcb_{W}(1)})-[C]=d_B(i_1^*(b)).
\]

Lemma \ref{lemrestA1=0} insures that $i_1^*(\LcLcub W)- i_1^*(\ol{\Lcb_{W}(1)})=0$.
 Thus
one has
\[
-[C]=d_B(i_1^*(b)) + \mx{ shuffle products}
\] 
 which shows that $[p^*(C)]$ is zero in $\HH^0(B_{\A^1})$
modulo shuffles. Hence Equation \eqref{eq:LcLcubW-C} can
be written has 
\[
\LcLcub W - \ol{\Lcb_W(1)}=0 \qquad \mx{in} \quad Q_{\HH^0(B_{\A^1})}=\HH^0(Q_{\A^1})
\]

Finally, taking the restriction to $\ps$,  one has $\Lcb_W-\Lcub_W=j^*(\LcLcub
W)$.

\end{proof}
The main consequence of Equation \eqref{eq:LcLcuBX} in the previous theorem is
that in $Q_{\HH^0(B_X)}$ one can replace the bar avatar of the geometric
differential system \eqref{ED-Lc} by a bar avatar of the differential system
\eqref{ED-T} coming from Ihara action by special derivations. 
\begin{coro} In $Q_{\HH^0(B_X)}$, the set of indecomposable elements of
  $\HH^0(B(\Nge{X}{\bullet}))$, the following holds for any (non-empty) Lyndon
  word $W$: 
\begin{equation} \tag{ED-$Q_X$}\label{ED-QX} 
 \delta_Q(\Lcb_{W})=\sum_{U<V}\alpha_{U,V}^W\Lcb_{U}\w \Lcb_{V} 
\sum_{U,V}\beta_{U,V}^W\Lcb_{U}\w\Lcb_{V}(1).
\end{equation}
\end{coro}
\begin{proof}
Let $W$ be a Lyndon word. The statement holds when $W$ has length equal 
$1$ and one can assume that $W$ has length greater or equal to $2$.
One begins with the formula giving $\delta_Q(\Lcb_W)$ from Theorem
\ref{bareltX}:
\[
\delta_Q(\Lcb_W)=\sum_{U<V} a_{U,V}^W\Lcb_U \w \Lcb_V 
\sum_{U,V}b_{U,V}^W\Lcb_{U}
\w\Lcub_V.
\]
Then using the relations given by Equation \eqref{eq:LcLcuBX}, one has 
\begin{equation*}
\delta_Q(\Lcb_W)=\sum_{U<V} a_{U,V}^W\Lcb_U\w \Lcb_V  
  \sum_{U,V}b_{U,V}^W\Lcb_{U}\w \left(\Lcb_V-\Lcb_V(1)\right).
\end{equation*}
Expanding terms as $\Lcb_{U}\left(\Lcb_V-\Lcb_V(1)\right) $ and we conclude the
proof using the
expression of coefficients $a$'s and $b$' in terms of
$\alpha$'s and $\beta$'s given at Lemma \ref{lem:aapbbpa}.
\end{proof}
\subsection{A Basis for the geometric Lie coalgebra}\label{subsec:relbasis}
This section shows that the image of the family of bar elements $\Lcb_W$ in
Deligne-Goncharov motivic fundamental Lie coalgebra is a basis of this coLie
coalgebra. Hence the family $\Lcb_W$ induced a basis of the tannakian coLie
coalgebra of mixed Tate motives over $\ps$ relative to the one for mixed Tate
motives over $\Q$.

We recall that for $X=\ps$, M. Levine in 
\cite{LEVTMFG}[Theorem 5.3.2 and beginning of the section 6.6]
shows, one can identify the Tannakian group associated with $\MTM(X)$ with
the spectrum of $\HH^0(B_X)$:
\[
G_{\MTM(X)}\simeq \Sp(\HH^0(B_X)).
\]

Then, he uses a relative bar-construction in order to relate $G_{MTM(X)}$ to the
motivic fundamental group of $X$ of Goncharov
and Deligne, $\pi_1^{mot}(X,x)$ (see \cite{GFPLDeli} and \cite{DG}).

\begin{thm}[{\cite{LEVTMFG}[Corollary 6.6.2]}]\label{thm:pi1exactseq} Let $x$ be a
  $\Q$-point of 
  $X=\ps$. Then there is a split exact sequence:

 \[
 \begin{tikzpicture}
\matrix (m) [matrix of math nodes,
 row sep=0.5em, column sep=2.3em, 
 text height=1.3ex, text depth=0.25ex] 
 {1 & {\pi_{1}^{mot}(X,x)} & {\Sp(\HH^0(B(\mc N_X)))} &{ \Sp(\HH^0(B(\mc N_{\Q})))} & 1
 \\};
 \path[->,font=\scriptsize]
 (m-1-1) edge node[auto] {} (m-1-2)
(m-1-2) edge node[auto] {} (m-1-3)
(m-1-3) edge node[auto] {$p^*$} (m-1-4)
(m-1-4) edge node[auto] {} (m-1-5);
\path[->,font=\scriptsize, bend left]
(m-1-4) edge node[auto] {$x^*$} (m-1-3);
 \end{tikzpicture}
 \]  
where $p$ is the structural morphism $p: \ps \lra \Sp(\Q)$.
\end{thm}

 Theorem \ref{thm:pi1exactseq} can be reformulate in
terms of Lie coalgebras, looking at indecomposable elements of the respective
Hopf algebras.
\begin{prop}
There is a split exact sequence of Lie coalgebras:
 \[
 \begin{tikzpicture}
\matrix (m) [matrix of math nodes,
 row sep=0.6em, column sep=3.5em, 
 text height=1.5ex, text depth=0.25ex] 
 {0 &Q_{\HH^0(B_{\Q})} & Q_{\HH^0(B_{X})} &Q_{geom} & 0
 \\};
 \path[->,font=\scriptsize]
 (m-1-1) edge node[auto] {} (m-1-2)
(m-1-2) edge node[auto] {$\tilde p$} (m-1-3)
(m-1-3) edge node[auto] {$\phi$} (m-1-4)
(m-1-4) edge node[auto] {} (m-1-5);
\path[->,font=\scriptsize, bend left]
(m-1-3.south west) edge node[auto] {$\tilde x$} (m-1-2.south east);
 \end{tikzpicture}
 \]  
where $Q_{geom}$ is the set of indecomposable elements of
$\mc{O}(\pi_{1}^{mot}{X,x})$ and is isomorphic as Lie coalgebra to the graded
dual of  the 
 Lie algebra associated to $\pi_1^{mot}(X,x)$. Hence, $Q_{geom}$ is isomorphic
as Lie coalgebra  to the graded dual of the free Lie algebra on two generators
$\Lie(X_0,X_1)$. 
\end{prop} 
 Considering the family of bar elements $\Lcb_W$ for all Lyndon words $W$ in
 this short exact sequence of Lie coalgebra,  ones gets
\begin{thm}\label{Lcbbasis}
The family $\phi(\Lcb_W)$ for any Lyndon words $W$ is a basis of the Lie
coalgebra $Q_{geom}$. Hence the family $\Lcb_W$ is a basis of $Q_{\HH^0(B_{X})}$
relatively to $Q_{\HH^0(B_{\Q})}$.
\end{thm} 
\begin{proof} The above short exact sequence being a sequence of Lie
  coalgebra, one has 
\[
 \begin{tikzpicture}
\matrix (m) [matrix of math nodes,
 row sep=2.5em, column sep=1.9em,
 text height=1.4ex, text depth=0.25ex] 
 {0 &Q_{\HH^0(B_{\Q})} \w Q_{\HH^0(B_{\Q})} & Q_{\HH^0(B_{X})} \w
   Q_{\HH^0(B_{X})} &Q_{geom} \w Q_{geom} & 0
 \\
0 &Q_{\HH^0(B_{\Q})} & Q_{\HH^0(B_{X})} &Q_{geom} & 0
 \\};
 \path[->,font=\scriptsize]
 (m-2-1) edge node[auto] {} (m-2-2)
(m-2-2) edge node[auto] {$\tilde p$} (m-2-3)
(m-2-3) edge node[auto] {$\phi$} (m-2-4)
(m-2-4) edge node[auto] {} (m-2-5)
(m-2-2) edge node [auto] {$\delta_{Q, \Q}$} (m-1-2)
(m-2-3) edge node [auto] {$\delta_{Q, X}$} (m-1-3)
(m-2-4) edge node [auto] {$\delta_{geom}$} (m-1-4)
 (m-1-1) edge node[auto] {} (m-1-2)
(m-1-2) edge node[auto] {$\tilde p \w \tilde p$} (m-1-3)
(m-1-3) edge node[auto] {$\phi\w \phi$} (m-1-4)
(m-1-4) edge node[auto] {} (m-1-5)
;
 \end{tikzpicture}
 \] 
As $\delta_{Q,X}(\Lcb_0)=\delta_{Q,X}(\Lcb_1)=0$, weight reasons show that
$\phi(\Lcb_0)$ and $\phi(\Lcb_1)$ are dual to the  weight $1$ generators  of $\Lie(X_0,X_1)$.

In order to show that the family $\phi(\Lcb_W)$ is a basis of $Q_{geom}$, it is
enough to show that the elements  $\phi(\Lcb_W)$ satisfy:
\[
\delta_{geom}(\phi(\Lcb_W))=\sum_{U<V} \alpha_{U,V}^W \phi(\Lcb_U) \w \phi(\Lcb_V)
\]
because $\delta_{geom}$ is dual to the bracket $[\, , \, ]$ of $\Lie(X_0,X_1)$.

As $\phi$ commutes with the cobracket, it is enough to compute $(\phi \w \phi)
\circ \delta_X(\Lcb_W)$ :
\begin{align*}
\delta_{geom}(\phi(\Lcb_W))=&(\phi \w \phi) \circ \delta_X(\Lcb_W) \\
=&(\phi\w \phi)\left(\sum_{U<V}\alpha_{U,V}^W\Lcb_{U}\w \Lcb_{V} 
\sum_{U,V}\beta_{U,V}^W\Lcb_{U}\w\Lcb_{V}(1) \right) \\[2em]
=&\sum_{U<V}\alpha_{U,V}^W\phi(\Lcb_{U})\w \phi(\Lcb_{V}) 
\sum_{U,V}\beta_{U,V}^W\phi(\Lcb_{U})\w\phi(\Lcb_{V}(1))
\end{align*}

By construction $\phi(\Lcb_{V}(1))$ is zero. Thus one obtains the expected formula for
$\delta_{geom}(\phi(\Lcb_W))$. 
\end{proof}

Note that $\delta_X$ gives the coaction of $Q_{\HH^0(B_{\Q})}$ on $Q_{geom}$
described in \cite{BrownMTMZ} in relation with Goncharov motivic coproduct
$\Delta^{mot}$. In this context, 
Equation \eqref{ED-QX}
\[
 \delta_Q(\Lcb_{W})=\sum_{U<V}\alpha_{U,V}^W\Lcb_{U}\w \Lcb_{V} 
\sum_{U,V}\beta_{U,V}^W\Lcb_{U}\w\Lcb_{V}(1)
\]
is nothing but another expression for Goncharov motivic cobracket $1/2(\Delta^{mot}-\tau
\Delta^{mot})$. This new expression has the advantage that it is stable under
the generating family $\Lcb_W$. 
\bibliographystyle{amsalpha}
\nocite{}
\bibliography{barbase}
\end{document}